\documentclass[preprint,11pt]{article}

\usepackage{amsfonts}
\usepackage{amsmath}
\usepackage{amssymb}
\usepackage{amstext}
\usepackage{amsthm}
\usepackage{mathrsfs}
\usepackage{graphicx}
\usepackage{subfigure}
\usepackage{colortbl,dcolumn}
\usepackage{comment}
\usepackage{psfrag}
\usepackage{booktabs}
\usepackage[nocompress]{cite}
\usepackage{enumerate}
\usepackage{extarrows}
\usepackage{epstopdf}

\usepackage{threeparttable}
\usepackage{caption}
\usepackage[colorlinks,
            linkcolor=black,
            anchorcolor=black,
            citecolor=black]{hyperref}

\numberwithin{equation}{section}
\usepackage[top=2cm, bottom=2cm, left=2cm, right=2cm]{geometry}

\newcommand{\R}{\mathbb{R}}
\newcommand{\N}{\mathbb{N}}
\newcommand{\E}{\mathbb{E}}
\newcommand{\LL}{\mathcal{L}} 
\newcommand{\D}{\mathcal{D}}
\newcommand{\F}{\mathcal{F}}
\renewcommand{\P}{\mathbb{P}}
\newcommand{\tr}{\operatorname{Tr}}

\newcommand{\dif}[1]{\mathrm{d}#1}%
\newcommand{\diff}[1]{\, \mathrm{d} #1}

\newtheorem{thm}{Theorem}[section]
\newtheorem{den}[thm]{Definition}
\newtheorem{lea}[thm]{Lemma}
\newtheorem{prn}[thm]{Proposition}
\newtheorem{coy}[thm]{Corollary}
\newtheorem{rek}[thm]{Remark}
\newtheorem{exe}[thm]{Example}
\newtheorem{asn}[thm]{Assumption}
\begin{document}

\title{
A full-discrete exponential Euler approximation of the invariant measure for parabolic stochastic partial differential equations\footnotemark[1]
}

\footnotetext{
  \footnotemark[1]This work was supported by NSF of China (11971488, 11571373, 11671405, 91630312),
  NSF of Hunan Province (2020JJ2040, 2018JJ3628)
  and Shenghua Yuying Program of Central South University.
             }

\author
    {
        Ziheng Chen$^{\text{a,b}}$, Siqing Gan$^{\text{a}}$, 
        Xiaojie Wang$^{\text{a}}$\footnotemark[2]
        \\
          {\small $^{\text{a}}$School of Mathematics and Statistics,
           Central South University, Changsha 410083, Hunan, China
          }
        \\
          {\small $^{\text{b}}$Institute of Computational Mathematics, Scientific/Engineering Computing,}
          \\
          {\small Academy of Mathematics and Systems Science, Chinese Academy of Sciences, Beijing 100190, China
          }
    }

\footnotetext
    {
    \footnotemark[2]Corresponding author:
               x.j.wang7@gmail.com (X. Wang).
    }

\maketitle

\begin{abstract}
  {
        \rm\small
    We discrete the ergodic semilinear stochastic partial differential equations in space dimension $d \leq 3$ with additive noise, spatially by a spectral Galerkin method and temporally by an exponential Euler scheme.
    It is shown that both the spatial semi-discretization and the spatio-temporal full discretization are ergodic. Further, convergence orders of the numerical invariant measures, depending on the regularity of noise, are recovered based on an easy time-independent weak error analysis without relying on Malliavin calculus. To be precise, the convergence order is $1-\epsilon$ in space and $\frac{1}{2}-\epsilon$ in time for the space-time white noise case and $2-\epsilon$ in space and $1-\epsilon$ in time for the trace class noise case in space dimension $d = 1$, with arbitrarily small $\epsilon>0$. Numerical results are finally reported to confirm these theoretical findings.
  }

\vspace{1em}

\vspace{1em}
\textbf{AMS subject classifications:}
    {
        \rm\small
        60H15, 60H35, 37M25
    }
\end{abstract}


\section{Introduction}

This work concerns the
semilinear stochastic partial differential equations (SPDEs)
\begin{equation}
      \label{eq:SPDE}
      \diff{X(t)}
      =
      A X(t) \diff{t}
      +
      F(X(t)) \diff{t}
      +
      \diff{W^{Q}(t)},
      \quad\forall\, t > 0,
      \quad
      X(0) = X_{0},
\end{equation}
where the dominant linear operator $A \colon \mathcal{D}(A) \subset H \to H$ generates
a strongly continuous semigroup $E(t) = e^{tA}, t \geq 0$ on a real separable Hilbert
space $(H, \langle \cdot,\cdot \rangle, \| \cdot \|)$ and $F \colon H \to H$ is a nonlinear deterministic mapping.
Moreover, $\{W^{Q}(t)\}_{t \geq0}$ is an $H$-valued (possibly cylindrical) $Q$-Wiener
process on a filtered probability space $(\Omega, \F, \P; \{ \F_{t} \}_{t \geq 0})$,
with the covariance operator $Q$ obeying
\begin{equation}
            \label{eq:intro:AQ-condition}
            \|(-A)^{\frac{\beta-1}{2}}Q^{\frac{1}{2}}\|_{\mathcal{L}_{2}(H)}
            < \infty ,
            \quad\text{for some}~
            \beta \in(0,1].
      \end{equation}
 Such a setting covers both space-time white noise in space dimension $ d = 1$ and trace class noise in multiple space dimension $ d \leq 3$.
Under Assumptions \ref{ass:main.ass} and \ref{ass:main.ass2} specified later,
a unique mild solution $\{X(t)\}_{t \geq 0}$ of \eqref{eq:SPDE} exists, given by
\begin{equation}
      \label{eq:SIPDE}
            X(t)
      =
            E(t)X_{0}
      +
            \int_{0}^{t}
            E(t-s) F(X(s))
            \diff{s}
      +
            \int_{0}^{t}
            E(t-s)
            \diff{W^{Q}(s)},
       \quad\forall\, t \geq 0,
       \quad\P\text{-a.s.}
\end{equation}
This mild solution $\{X(t)\}_{t \geq 0}$ is shown to be ergodic
(see Section \ref{sec:setting} below for the precise definition of ergodicity), i.e.,
it admits a unique invariant probability measure $\nu$ on $(H,\mathcal{B}(H))$ such that
\begin{equation}\label{eq:ergodic.one}
      \lim\limits_{T\to\infty}
      \frac{1}{T}
      \int_{0}^{T} \E\big[\Phi(X(t))\big] \diff{t}
      =
      \int_{H} \Phi(y) \,\nu(\dif{y})
      \quad\text{in}~L^{2}(H,\nu),
      \quad \forall\, \Phi \in L^{2}(H,\nu).
\end{equation}
The ergodicity characterizes the longtime behaviour of the considered equation and
has significant impacts on quantum mechanics, fluid dynamics, financial mathematics
and many other scientific fields \cite{da1996ergodicity}.
In many applications, it is desirable to compute the mean of a given function with respect to the invariant law of the diffusion, i.e., the ergodic limit $\int_{H} \Phi(y) \,\nu(\dif{y})$. To this end, one often has to integrate a system over comparatively long time intervals, which is one of the most serious difficulties from the computational point of view. Moreover, it is usually impossible to exactly simulate the ergodic limit for a nonlinear system since the explicit expression of $\nu$ is rarely available and the support of $\nu$ is an infinite-dimensional space.
This motivates recent interests in
developing and analyzing numerical schemes that can inherit the ergodicity of the original system
and can approximate the ergodic limit efficiently.

For finite dimensional stochastic differential equations (SDEs),
much progress has been made in the design and analysis of approximations of invariant measures
(see,~e.g.,~\cite{talay1990second,mattingly2002ergodicity,
talay2002stochastic,milstein2007computing,
mattingly2010convergence,abdulle2014high,chen2019ergodic,lu2019approximation} and other references therein).
%
By contrast, approximations of invariant measures for SPDEs are at an early stage and just a very limited number of literature
\cite{brehier2014approximation,brehier2016high,
brehier2017approximation,chen2017approximation,hong2017numerical,hong2019invariantbook}
are devoted to this topic.
In 2014, Br{\'e}hier \cite{brehier2014approximation} first studied the temporal semi-discretization by the linear implicit Euler scheme for
semilinear SPDEs of parabolic type driven by additive space-time white noise. To achieve higher order accuracy,
Br{\'e}hier and Vilmart \cite{brehier2016high} further introduced a kind of implicit-explicit postprocessed method for the temporal semi-discretization. In the more recent publication \cite{brehier2017approximation}, Br{\'e}hier and Kopec analyzed spatio-temporal full discretizations by the finite element method and the semi-implicit Euler scheme to approximate invariant measures for the semi-linear SPDEs with additive
space-time white noise.
Instead of the linear implicit Euler scheme used in \cite{brehier2014approximation,brehier2017approximation},
we turn our attention to an exponential Euler type fully discrete scheme for approximating invariant measures of more general SPDEs \eqref{eq:SPDE}.
It is known that exponential integrators, as explicit time-stepping schemes, are successfully used to solve deterministic stiff problems such as
parabolic partial differential equations and their spatial discretizations (see the
survey article \cite{hochbruck2010exponential} and references therein).
The extension to SPDEs has been extensively studied
in \cite{jentzen2009overcoming,jentzen2011efficient,kloeden2011exponential,lord2013stochastic,wang2015exponential,
tambue2016weak,wang2016weak,anton2017fully,anton2016full,cohen2013trigonometric,cohen2016fully},
where both strong and weak convergence of exponential integrators were well established for SPDEs over finite time intervals.
However, the weak error analysis in infinite horizon for exponential integrators is missing,
which partly motivates this work.
Actually,  in the present article we analyze convergence orders of the numerical invariant measures done by the exponential Euler scheme,
based on an easy time-independent weak error analysis without relying on Malliavin calculus,
which is required in the analysis of \cite{brehier2014approximation,brehier2017approximation} for the linear implicit Euler scheme.
From the point of view of computational implementation, we find that both the linear implicit Euler scheme and the exponential Euler scheme
can be explicitly implemented due to a spectral Galerkin spatial discretization here and thus spend essentially the same computational costs. But the exponential Euler scheme is always considerably more accurate than the linear implicit Euler scheme for various noises and time stepsizes, as clearly indicated by the numerical results in Table \ref{tab:temporalweakerrors} of Section \ref{sect:numer-exp}.
We first discrete \eqref{eq:SPDE} in space by a spectral Galerkin method
\begin{equation}\label{eq:SDE.simplify}
            \diff{X^{n}(t)}
            =
            A_{n} X^{n}(t) \diff{t}
            +
            P_{n}F(X^{n}(t)) \diff{t}
            +
            P_{n} \diff{W^{Q}(t)},
            \quad\forall\, t > 0,
            \quad
           X^{n}(0) =P_{n}X_{0},
\end{equation}
where $P_{n}$ is a projection operator from $H$ to the finite-dimensional space $H_{n} \subset H, n \in \N$ and $A_{n} :=  A P_{n}$ is a bounded linear operator on $H_{n}$ (see Subsection \ref{sec:discrete.semi} below for precise description).
As pointed out in \cite[Chapter 2]{lord2014introduction}, the spectral Galerkin method
is particularly suitable for simple domains and smooth data and trivial to compute in any situation where the eigenfunctions are explicitly known. Therefore we restrict ourselves to the spectral Galerkin method \eqref{eq:SDE.simplify} but we could also consider the finite element spatial discretization as explained in \cite[Remark 1]{wang2016weak}.
Observing that \eqref{eq:SDE.simplify} is a finite-dimensional SDE in $H_{n}$ (or equivalently in $\R^n$), we apply a general ergodicity theory established in \cite{da2006introduction}
to verify the ergodicity of $\{X^{n}(t)\}_{t \geq 0}$, which possesses a unique invariant measure $\nu^{n}$.
Further, we carry out the time-independent weak error analysis, thanks to the uniform boundedness of the mean square moment of $\{X^{n}(t)\}_{t \geq 0}$ and the improved regularity for the associated Kolmogorov equation (see Theorem \ref{th:exist.unique} and Proposition \ref{pro:regularityDvntx} below). Then the ergodicity and the time-independent weak error
of $\{X(t)\}_{t \geq 0}$ and $\{X^{n}(t)\}_{t \geq 0}$ help us to derive the error between $\nu$ and $\nu_{n}$, given by
\begin{equation}
            \label{eq:intro-order.measure.spatial}
            \Big|
            \int_{H} \Phi(y) \,\nu(\dif{y})
            -
            \int_{H_{n}} \Phi(y) \,\nu^{n}(\dif{y})
            \Big|
            \leq
            C\lambda_{n}^{-\beta+\epsilon}.
      \end{equation}
Here  $\epsilon>0$ is arbitrarily small,  $\beta \in (0, 1]$ comes from \eqref{eq:intro:AQ-condition} and $\lambda_{n}$ serves as the $n$-th eigenvalue of the linear operator $- A$.
Given $\tau > 0$ being the uniform time stepsize, the exponential Euler scheme takes the form of
\begin{equation}
      \label{eq:EES.simplify}
      Y_{m+1}^{n}
      =
      E_{n}(\tau)Y_{m}^{n}
      +
      {\tau}E_{n}(\tau)P_{n}F(Y_{m}^{n})
      +
      E_{n}(\tau)P_{n}\Delta{W_{m}^{Q}},
      \quad
      Y_{0}^{n} = X_{0}^{n},
\end{equation}
where $Y_{m}^{n}$ is the numerical approximation of $X^{n}(t_m)$.
As one of the key ingredients to guarantee the ergodicity and nice regularity of $\{Y_{m}^{n}\}_{m \in \N}$,
the semigroup operator $E_{n}(\tau) = e^{\tau A_n}$ exhibits an exponentially decreasing property in the sense of $\| E_{n}(\tau) \|_{\mathcal{L}(H_{n})} \leq e^{-\lambda_{1}\tau}$, $\lambda_{1} > 0,\tau > 0$. More formally, we rely on the general ergodicity theory of Markov chain established in \cite{mattingly2002ergodicity} to show the ergodicity of $\{Y_{m}^{n}\}_{m \in \N}$, with a unique invariant measure $\nu_{\tau}^{n}$.
Now it remains to do the time-independent weak error analysis of the temporal discretization, which starts from
a weak error representation formula presented in \cite{wang2016weak}.
There the weak error analysis was done on a finite time interval $[0,T]$.
However, weak error estimates here must be time-independent and hold over long time.
%
Again, owing to the ergodicity and time-independent weak error of $\{X^{n}(t)\}_{t \geq 0}$ and $\{Y_{m}^{n}\}_{m \in \N}$, the error between $\nu^{n}$ and $\nu_{\tau}^{n}$ can be measured as
\begin{equation}\label{eq:intro-order.measure.space}
\Big|
            \int_{H_{n}} \Phi(y) \,\nu^{n}(\dif{y})
            -
            \int_{H_{n}} \Phi(y) \,\nu_{\tau}^{n}(\dif{y})
            \Big|
            \leq
            C{\tau}^{\beta-\epsilon}.
      \end{equation}
Combining this with \eqref{eq:intro-order.measure.spatial} results in the space-time full approximations of invariant measures
(Corollary \ref{Corollary:full-discrete}).
Specializing  \eqref{eq:intro-order.measure.spatial} and \eqref{eq:intro-order.measure.space} into the case of
space dimension $d = 1$, implies that
    the convergence order is $1-\epsilon$ in space and $\frac{1}{2}-\epsilon$ in time for the space-time white noise case
    and $2-\epsilon$ in space and $1-\epsilon$ in time for the trace class noise case, with arbitrarily small $\epsilon>0$.

To conclude, 
convergence orders of the  numerical invariant measures, depending on the regularity of noise,
are recovered based on an easy time-independent weak error analysis without relying on Malliavin calculus,
which is required in the analysis of \cite{brehier2014approximation,brehier2017approximation} for the linear implicit Euler scheme.
Furthermore, numerical results reveal that the exponential Euler scheme performs better than the linear implicit Euler scheme.
Finally we mention that one can consult \cite{chen2017approximation,hong2017numerical,hong2019invariantbook} for recent progress on approximations of invariant measures for stochastic nonlinear Schr\"{o}dinger equations and \cite{andersson2016duality,andersson2016weak,
brehier2018kolmogorov,brehier2018weak,conus2019weak,cox2019weak,
debussche2011weak,jentzen2015weak,wang2013weak}
and references therein, for other relevant works on weak approximations over finite time intervals.

%




The rest of this paper is organized as follows. Some setting and assumptions are collected in the next section.
Sections 3 and 4 focus on the ergodicity of the numerical approximations for both spatial and temporal discretizations
as well as the error estimates between invariant measures.
Numerical experiments are finally performed to illustrate the theoretical results in Section 5.

\section{Setting and assumptions}
        \label{sec:setting}

Throughout this paper, we need the following notation.
Let $C$ be a generic constant that may vary from one place to another.
Let $\N = \{1,2,\ldots\}$ be the set of positive integers and
$\epsilon > 0$ be an arbitrarily small parameter.
Let $(H, \langle \cdot,\cdot \rangle_{H}, \|\cdot\|_{H})$ and
$(U, \langle \cdot,\cdot \rangle_{U}, \|\cdot\|_{U})$ be two real separable Hilbert spaces.
By $C_{b}^{k}(U,H)$ we denote the space of not necessarily bounded mappings from $U$ to $H$ that have continuous and bounded Fr\'{e}chet derivatives up to order $k$ for $k = 1,2$.
Furthermore, by $\mathcal{L}(U,H)$ we denote the space of all bounded linear operators from $U$ to $H$ with the usual operator norm $\|\cdot\|_{\mathcal{L}(U,H)}$ and write $\mathcal{L}(U) := \mathcal{L}(U,U)$ for simplicity.
Moreover, we need the space of all nuclear operators from $U$ to $H$ denoted by $\mathcal{L}_{1}(U,H)$ and the space of all Hilbert--Schmidt operators from $U$ to $H$ by $\mathcal{L}_{2}(U,H)$.
Analogously, we write $\mathcal{L}_{1}(U) := \mathcal{L}_{1}(U,U)$ and $\mathcal{L}_{2}(U) := \mathcal{L}_{2}(U,U)$.
As usual, $\mathcal{L}_{1}(U)$ and $\mathcal{L}_{2}(U,H)$ are endowed with the nuclear norm $\|\cdot\|_{\mathcal{L}_{1}(U)}$ and the Hilbert--Schmidt norm $\|\cdot\|_{\mathcal{L}_{2}(U,H)}$, respectively,
\begin{equation}\label{eq:twonorms}
      \|\Gamma_{1}\|_{\mathcal{L}_{1}(U)}
       =
      \sum\limits_{i=1}^{\infty}
      \langle \Gamma_{1}\psi_{i},\psi_{i} \rangle,
      \quad
      \|\Gamma_{2}\|_{\mathcal{L}_{2}(U,H)}
      =
      \Big(
      \sum\limits_{i=1}^{\infty}
      \| \Gamma_{2}\psi_{i} \|_{H}^{2}
      \Big)^{\frac{1}{2}}
\end{equation}
for any $\Gamma_{1} \in \mathcal{L}_{1}(U)$ and $\Gamma_{2} \in \mathcal{L}_{2}(U,H)$.
Additionally, the norms defined in \eqref{eq:twonorms} do not depend on the particular choice of the orthonormal basis $\{\psi_{i}\}_{i \in \N}$ of $U$, see \cite[Appendix~C]{da1992stochastic}.
For the convenience of the following analysis, we list some norm inequalities, see \cite[Appendix~B]{prevot2007concise}.
If $\Gamma_{1} \in \mathcal{L}_{1}(U)$ and
$\Gamma_{2} \in \mathcal{L}(U)$, then
$\Gamma_{1}^{*} \in \mathcal{L}_{1}(U)$,
$\Gamma_{1}\Gamma_{2} \in \mathcal{L}_{1}(U)$,
$\Gamma_{2}\Gamma_{1} \in \mathcal{L}_{1}(U)$ and
\begin{equation}
      \label{eq:.operator.Gamma.one}
      |\tr\Gamma_{1}| \leq \|\Gamma_{1}\|_{\mathcal{L}_{1}(U)},
      \quad
      \tr(\Gamma_{1}^{*}) = \tr(\Gamma_{1}),
      \quad
      \tr(\Gamma_{1}\Gamma_{2}) = \tr(\Gamma_{2}\Gamma_{1}).
\end{equation}
When $\Gamma_{1} \in \mathcal{L}_{2}(U,H)$ and
$\Gamma_{2} \in \mathcal{L}_{2}(H,U)$, it holds that
$\Gamma_{1}^{*} \in \mathcal{L}_{2}(H,U)$,
$\Gamma_{1}\Gamma_{2} \in \mathcal{L}_{1}(H)$ and
\begin{equation}
      \label{eq:.operator.Gamma.two}
      \|\Gamma_{1}^{*}\|_{\mathcal{L}_{2}(H,U)}
      =
      \|\Gamma_{1}\|_{\mathcal{L}_{2}(U,H)},
      \quad
      \|\Gamma_{1}\Gamma_{2}\|_{\mathcal{L}_{1}(H)}
      \leq
      \|\Gamma_{1}\|_{\mathcal{L}_{2}(U,H)}
      \|\Gamma_{2}\|_{\mathcal{L}_{2}(H,U)}.
\end{equation}
For $\Gamma \in \mathcal{L}(U,H)$ and
$\Gamma_{j} \in \mathcal{L}_{j}(U), j = 1,2$,
we have $\Gamma\Gamma_{j} \in \mathcal{L}_{j}(U,H)$ and
\begin{equation}
      \label{eq:.operator.Gamma.three}
      \|\Gamma\Gamma_{j}\|_{\mathcal{L}_{j}(U,H)}
      \leq
      \|\Gamma\|_{\mathcal{L}(U,H)}
      \|\Gamma_{j}\|_{\mathcal{L}_{j}(U)},~~j = 1,2.
\end{equation}

To proceed, we need some assumptions.

\begin{asn}
      \label{ass:main.ass}
      Let $A \colon \D(A) \subset H \to H$ be a densely defined, self-adjoint, negative definite linear operator, which is not necessarily bounded but with compact inverse.
\end{asn}
In the above setting, the dominant linear operator $A$ generates a strongly continuous semigroup of contractions $E(t) = e^{tA}, t \geq 0$ on $H$ and
there exists an increasing sequence of real numbers
$\{\lambda_{i}\}_{i \in \N}$ and an orthonormal basis
$\{e_{i}\}_{i \in \N}$ of $H$ such that
\begin{equation}\label{eigenpairs}
-Ae_{i} = \lambda_{i}e_{i},
\quad
\forall \, i \in \N
\quad
\mbox{ with }
\quad
    0 < \lambda_{1} \leq \lambda_{2}
    \leq \cdots \leq \lambda_{n}
    \to \infty
    ~\text{as}~ n \to \infty.
\end{equation}
This allows us to define fractional powers of $-A$, i.e., $(-A)^{\gamma}, \gamma \in \R$, in a much simple way, see \cite[Appendix~B.2]{kruse2014strong}.
So we introduce the Hilbert space
     $
         \dot{H}^{\gamma}
     =
         \mathcal{D}((-A)^{\frac{\gamma}{2}})
     $
for every $\gamma \in \R$, equipped with the inner product
     $
         \langle \varphi,\psi \rangle_{\dot{H}^{\gamma}}
     =
         \big\langle
             (-A)^{\frac{\gamma}{2}}\varphi
             ,
             (-A)^{\frac{\gamma}{2}}\psi
         \big\rangle
     =
         \sum_{i=1}^{\infty}
         \lambda_{i}^{\gamma}
         \langle \varphi,e_{i} \rangle
         \langle \psi,e_{i} \rangle
     $
and the corresponding norm
    $
         \|\varphi\|_{\gamma}
    =
         \sqrt{\langle \varphi,\varphi \rangle_{\dot{H}^{\gamma}}}
    $
for all $\varphi,\psi \in \dot{H}^{\gamma}$.
The next lemma gives some smoothing properties of semigroup $\{E(t)\}_{t \geq 0}$, see \cite[Proposition 2.4]{brehier2014approximation} and \cite[Proposition 2.6]{brehier2017approximation} for similar results. Since we will make use of them very frequently, we present a proof, but only under the above assumption on $A$.
\begin{lea}\label{lea:semigroupsmooth}
	Suppose that Assumption \ref{ass:main.ass} holds. Then there exists $C>0$ such that
	\begin{equation}\label{lem:semigroup.smooth}
	\begin{split}
	\|(-A)^{\gamma}E(t)\|_{\mathcal{L}(H)}
	\leq&
	Ct^{-\gamma}e^{-\frac{\lambda_{1}}{2}t},
	\quad\forall\,
	t > 0,\gamma \geq 0,
	\\
	\|(-A)^{-\rho}(E(t) - E(s))\|_{\mathcal{L}(H)}
	\leq&
	C(t-s)^{\rho}e^{-\frac{\lambda_{1}}{2}s},
	\quad\forall\,
	0 \leq s < t , \rho \in [0,1].
	\end{split}
	\end{equation}
\end{lea}

\begin{proof}
	For the first part of \eqref{lem:semigroup.smooth}, we use \eqref{eigenpairs} and the spectral mapping theorem \cite[Section 3.2]{sell2002dynamics} to show
	\begin{equation*}
		\|(-A)^{\gamma}E(t)\|_{\mathcal{L}(H)}
		=
		\sup\limits_{i \in \N}\|(-A)^{\gamma}E(t)e_{i}\|
		=
        \sup\limits_{i \in \N}\lambda_{i}^{\gamma}e^{-\lambda_{i}t}
		\leq
        2^{\gamma}t^{-\gamma}e^{-\frac{\lambda_{1}}{2}t}\sup\limits_{i \in \N}
        \big(\tfrac{\lambda_{i}t}{2}\big)^{\gamma}e^{-\frac{\lambda_{i}t}{2}},
	\end{equation*}
	which leads to the desired result thanks to the fact that the function $x \mapsto x^{\gamma}e^{-x}$ is bounded for all $x \geq 0$ and $\gamma \geq 0$.
	It remains to verify the second part of \eqref{lem:semigroup.smooth}. Repeating the previous techniques yields
	\begin{equation*}
	\begin{split}
		&\|(-A)^{-\rho}(E(t) - E(s))\|_{\mathcal{L}(H)}
		=
		\sup\limits_{i \in \N}
		\|(-A)^{-\rho}(E(t) - E(s))e_{i}\|
		\\=&
		\sup\limits_{i \in \N}
		\lambda_{i}^{-\rho}(1 - e^{-\lambda_{i}(t-s)})e^{-\lambda_{i}s}
		\leq
		e^{-\frac{\lambda_{1}}{2}s}\sup\limits_{i \in \N}
        \lambda_{i}^{-\rho}(1 - e^{-\lambda_{i}(t-s)})e^{-\frac{\lambda_{i}}{2}s}
		\\\leq&
        (t-s)^{\rho}e^{-\frac{\lambda_{1}}{2}s}\sup\limits_{i \in \N}
        (\lambda_{i}(t-s))^{-\rho}(1 - e^{-\lambda_{i}(t-s)}).
    \end{split}
	\end{equation*}
	By the boundedness of the function $x \mapsto x^{-\rho}(1-e^{-x})$ for all $x\geq0$, $\rho \in [0,1]$, we complete the proof.
\end{proof}

%
%
\begin{asn}
\label{ass:main.ass2}
      Let $\{W^{Q}(t)\}_{t \geq 0}$ be a cylindrical $Q$-Wiener process on a filtred probability space $(\Omega, \F, \P$;
      $\{ \F_{t} \}_{t \geq 0})$ with $Q \colon H \to H$ being a self-adjoint, positive definite bounded linear operator.
      Furthermore, let $A$ and $Q$ be commutable and satisfy
      \begin{equation}
            \label{eq:A.and.Q}
            \|(-A)^{\frac{\beta-1}{2}}Q^{\frac{1}{2}}\|_{\mathcal{L}_{2}(H)}
            < \infty ,
            \quad\text{for some}~
            \beta \in(0,1].
      \end{equation}
      In addition, let the initial data $X_{0} \in \dot{H}^{\max( 2 \beta, 1 )}$ be deterministic.
      Let the nonlinear mapping $F \colon H \to H$ satisfy a one-sided Lipschitz condition
      \begin{equation}\label{eq:disipativity}
            \langle \varphi_{1}-\varphi_{2},F(\varphi_{1})-F(\varphi_{2}) \rangle
            \leq
            L_{F}\|\varphi_{1}-\varphi_{2}\|^{2},
            \quad\text{with}\quad
            L_{F} < \lambda_{1},
            \quad
            \forall\,\varphi_{1},\varphi_{2} \in H,
      \end{equation}
      where $\lambda_{1}$ is the smallest eigenvalue of $-A$. Finally, let $F$ be twice differentiable and there exist $\delta \in [1,2)$ and $\eta \in[0,1)$ such that
      \begin{align}
            \label{eq:F.derivative.one}
            \|F'(\varphi)\psi\|
            \leq&
            L\|\psi\|,\quad\forall\, \varphi,\psi \in H,
            \\
            \label{eq:F.derivative.one.delta}
            \|(-A)^{-\frac{\delta}{2}}F'(\varphi)\psi\|
            \leq&
            L(1+\|\varphi\|_{1})\|\psi\|_{-1},
            \quad\forall\, \varphi \in \dot{H}^{1}, \psi \in H,
            \\
            \label{eq:F.derivative.second}
            \|(-A)^{-\eta}F''(\varphi)(\psi_{1},\psi_{2})\|
            \leq&
            L\|\psi_{1}\|\|\psi_{2}\|,
            \quad\forall\, \varphi,\psi_{1},\psi_{2} \in H.
      \end{align}
\end{asn}


\begin{rek}
		We would like to mention that \eqref{eq:F.derivative.one} implies a globally Lipschitz condition on $F$
		and that one could simply use
		\eqref{eq:F.derivative.one}
		to obtain a one-sided Lipschitz condition like \eqref{eq:disipativity} but with a different one-sided Lipschitz constant $L$, i.e.,
		\begin{equation*}
           \langle \varphi_{1}-\varphi_{2},F(\varphi_{1})-F(\varphi_{2}) \rangle
            \leq
            L\|\varphi_{1}-\varphi_{2}\|^{2},
               \quad
            \forall\,\varphi_{1},\varphi_{2} \in H,
		\end{equation*}
instead of formulating \eqref{eq:disipativity} additionally.
In this way one would require $L < \lambda_1$, to ensure \eqref{eq:dissipativity.corollary} below, which in turn promises
the ergodicity of \eqref{eq:SIPDE}. Such a restriction, as required by \cite{brehier2016high,brehier2017approximation},
turns to be stricter than $L_F < \lambda_1$ because for some nonlinear mappings
the Lipschitz constant $L$ in \eqref{eq:F.derivative.one} might be very large
but the one-sided Lipschitz constant $L_F$
coming from \eqref{eq:disipativity} could be small (even negative).
%
%
For example, Nemytskii operators $F(\varphi)(\cdot) = f(\varphi(\cdot))$ with $f(x) = - \vartheta x$ for $\vartheta > \lambda_1$
satisfy \eqref{eq:F.derivative.one} and \eqref{eq:disipativity} with $L = \vartheta > \lambda_1$
and $L_F = - \vartheta < 0 < \lambda_1$, respectively.
To sum up, the condition \eqref{eq:disipativity} is used to relax the restriction for the ergodicity and the conditions \eqref{eq:F.derivative.one}--\eqref{eq:F.derivative.second} are required in the following longtime weak error analysis.
\end{rek}

It is well-known that $\{W^{Q}(t)\}_{t \geq 0}$ can be represented as
      \begin{equation}\label{eq:W.Q.t}
            W^{Q}(t)
            =
            \sum\limits_{i=1}^{\infty}
            \sqrt{q_{i}}\beta_{i}(t)e_{i},
            \quad\forall\,t \geq 0,
      \end{equation}
where $\{\beta_{i}(t)\}_{t \geq 0}$ for
$i \in \{n \in \N \colon q_{n} > 0\}$
are independent real-valued Brownian motions on $(\Omega, \F, \P)$ with respect to the filtration $\{ \F_{t} \}_{t \geq 0}$.
A class of semilinear stochastic heat equations satisfying the above assumptions can be found in \cite[Example~3.2]{wang2016weak}.
Moreover, under Assumptions \ref{ass:main.ass} and \ref{ass:main.ass2}, \eqref{eq:SPDE} admits a unique mild solution, see \cite[Theorem~5.3.1]{da1996ergodicity}.
\begin{thm}[\textbf{Existence, uniqueness of mild solution}]
\label{th:EueSPDEs}
      Suppose that Assumptions \ref{ass:main.ass} and \ref{ass:main.ass2} hold. Then  \eqref{eq:SPDE} admits a unique mild solution $\{X(t)\}_{t\geq0}$ given by \eqref{eq:SIPDE}.
\end{thm}

In the sequel we will introduce some concepts related to the ergodicity of $\{X(t)\}_{t\geq0}$.
By $B_b(H)$ (resp. $C_b(H)$) we denote the Banach space of all Borel bounded mappings (resp. continuous and bounded mappings) $\Phi \colon H \to \R$ endowed with the norm $\|\Phi\|_0 = \sup_{x \in H} |\Phi(x)|$. With this, we define the transition semigroup $P \colon [0,\infty) \to \LL(B_b(H))$ by
$$
  P_t\Phi(x) = \E\big[\Phi(X(t,x))\big],
  \quad
  \forall\, \Phi \in B_b(H),
$$
where $X(t,x)$ is the mild solution of \eqref{eq:SPDE} with $X(0) = x \in H$. Then it is easy to check that $\{P_t\}_{t \geq 0}$ is a Markov semigroup on $B_b(H)$, see\cite[Definition 5.1]{da2006introduction} for the precise definition of Markov semigroup.

Let us give some definitions related to $\{P_t\}_{t \geq 0}$. $\{P_t\}_{t \geq 0}$ is said to be strong Feller if $P_{t} \Phi \in C_b(H)$ for any $\Phi \in B_b(H)$ and any $t > 0$. Also, $\{P_t\}_{t \geq 0}$ is said to be irreducible if $P_{t} 1_{B(x_0,r)}(x) > 0$ for any $x,x_0 \in H,r>0$ and any $t > 0$, where $B(x_0,r)$ is the open ball in $H$ with center $x_0$ and radius $r>0$. Moreover, a probability measure $\mu$ on $(H,\mathcal{B}(H))$ is said to be invariant for $\{P_t\}_{t \geq 0}$ if
$$
  \int_{H} P_t \Phi \diff{\mu}
  =
  \int_{H} \Phi \diff{\mu},
  \quad
  \forall\, \Phi \in B_b(H), t \geq 0.
$$
According to the Von Neumann mean ergodic theorem \cite[Theorem 5.12]{da2006introduction}, the limit
$$
      \lim\limits_{T \to \infty}
      \frac{1}{T}
      \int_{0}^{T} P_{t} \Phi \diff{t},
      \quad
      \forall\, \Phi \in L^{2}(H,\mu)
$$
always exists in $L^{2}(H,\mu)$, where $L^{2}(H,\mu)$ is the space of all square integrable functions $\Phi \colon H \to \R$ with respect to $\mu$.

\begin{den}
Let $\mu$ be an invariant probability measure for $\{P_t\}_{t \geq 0}$.
We say that $\{P_t\}_{t \geq 0}$ is ergodic if
\begin{equation*}
      \label{eq:ergodic}
      \lim\limits_{T \to \infty}
      \frac{1}{T}
      \int_{0}^{T} P_{t} \Phi \diff{t}
  =
      \int_{H} \Phi(y) \,\mu(\dif{y})
  \quad
       \text{in~}
       L^{2}(H,\mu),
  \quad\forall\, \Phi \in L^{2}(H,\mu).
\end{equation*}
\end{den}

Now we say that the stochastic process $\{X(t,x)\}_{t\geq0}$ is ergodic if the associated Markov semigroup $\{P_t\}_{t \geq 0}$ is ergodic.
For any $\varphi \in \mathcal{D}(A) \subset H$, we can use \eqref{eigenpairs} and \eqref{eq:disipativity} to derive
\begin{equation*}
\begin{split}
      \langle A\varphi + F(\varphi),\varphi \rangle
      =&
      \Big\langle A\sum\limits_{i=1}^{\infty} \langle \varphi,e_i \rangle e_i,\sum\limits_{j=1}^{\infty} \langle \varphi,e_j \rangle e_j \Big\rangle
      +
      \langle F(\varphi)-F(0),\varphi \rangle
      +
      \langle F(0),\varphi \rangle
      \\\leq&
      -\lambda_{1}\|\varphi\|^{2} + L_F\|\varphi\|^{2} + \|F(0)\|\|\varphi\|
      \leq
      -\tfrac{\lambda_{1}-L_F}{2}\|\varphi\|^{2}
      +
      \tfrac{\|F(0)\|^2}{2(\lambda_{1}-L_F)},
\end{split}
\end{equation*}
where we have used the Cauchy--Schwarz inequality and the weighted Young inequality $ab \leq \varepsilon{a^{2}} + \frac{b^{2}}{4\varepsilon}$ for all $a,b \in \R$ with $\varepsilon = \frac{\lambda_{1}-L_F}{2} > 0$.
That is to say, we have
\begin{equation}\label{eq:dissipativity.corollary}
      \langle A\varphi + F(\varphi),\varphi \rangle
      \leq
      -c\|\varphi\|^{2} + C,
      \quad\forall\, \varphi \in \mathcal{D}(A)
\end{equation}
for some constants $c,C > 0$, which is a sufficient condition for $\{X(t,x)\}_{t\geq0}$ being ergodic, see, e.g., \cite[Section~8.6]{da1996ergodicity} and \cite{brehier2016high} for more details.
\begin{thm}[\textbf{Ergodicity of mild solution}]\label{th:EueSPDEs}
      Suppose that Assumptions \ref{ass:main.ass} and \ref{ass:main.ass2} hold. Then $\{X(t)\}_{t\geq0}$ given by \eqref{eq:SIPDE} is ergodic with a unique invariant probability measure $\nu$ satisfying \eqref{eq:ergodic.one}.
\end{thm}

\section{Spatial discretization and its ergodicity}

This section aims to analyze the error of invariant measures in the spatial direction.
To this end, we first obtain a numerical solution $\{X^{n}(t)\}_{t \geq 0}$ in space by applying a spectral Galerkin method to \eqref{eq:SPDE} 
in Subsection \ref{sec:discrete.semi}.
Subsection \ref{sec:ergodicity.semi} shows that $\{X^{n}(t)\}_{t \geq 0}$ is ergodic with a unique invariant measure $\nu_{n}$.
This ergodicity and the time-independent weak error established in
Subsection~\ref{eq:weak.error.semi} finally imply the convergence order of invariant measures $\nu$ and $\nu_{n}$ in Subsection~\ref{sec:order.semi}.

\subsection{Spectral Galerkin method}
           \label{sec:discrete.semi}

For every $n \in \N$, we define the finite-dimensional subspace $H_{n}$ of $H$ by
$
   H_{n} := \text{span}\{e_{1}, e_{2}, \ldots, e_{n}\}
$
and projection operator $P_{n} \colon H \to H_{n}$ by
$
      P_{n} \varphi
      =
      \sum_{i=1}^{n}
      \langle e_{i} , \varphi \rangle e_{i}
$
for all $\varphi \in H$. Now we introduce the spectral Galerkin approximation to \eqref{eq:SPDE} in $ H_{n} $ as follows
\begin{equation}\label{eq:SDE}
      \left\{
      \begin{array}{l}
      \hspace{-0.6em}
            \diff{X^{n}(t)}
            =
            A_{n} X^{n}(t) \diff{t}
            +
            P_{n}F(X^{n}(t)) \diff{t}
            +
            P_{n} \diff{W^{Q}(t)},
            \quad\forall\,t > 0,
      \\
      \hspace{-0.6em}
            X^{n}(0) = X_{0}^{n}:=P_{n}X_{0} \in H_{n},
      \end{array}
      \right.
\end{equation}where $A_{n} \colon H_{n} \to H_{n}$ is defined by
$A_{n} := A P_{n}$ and generates a strongly continuous semigroup
$E_{n}(t) = e^{tA_{n}}, t \geq 0$ on $H_{n}$.
Similarly, for every $\gamma \in \R$ we can define
$(-A_{n})^{\gamma} \colon H_{n} \to H_{n}$  as
$
    (-A_{n})^{\gamma}\varphi
    :=
    \sum_{i=1}^{n}
    \lambda_{i}^{\gamma}
    \langle \varphi,e_{i} \rangle e_{i}
$
for all $\varphi \in H_{n}$.
Note that $(-A_{n})^{\gamma}P_{n}\varphi = (-A)^{\gamma}P_{n}\varphi$ and
$E_{n}(t)P_{n}\varphi = E(t)P_{n}\varphi$ hold for all $\varphi \in H$.
Furthermore, variants of conditions in Assumptions~\ref{ass:main.ass}, \ref{ass:main.ass2} and
\eqref{lem:semigroup.smooth} remain true and will be frequently used in the following estimates. For example, we have
\begin{align}
      \label{eq:A.and.Q.n}
      \|
            (-A_{n})^{\frac{\beta-1}{2}}P_{n}Q^{\frac{1}{2}}
      \|_{\mathcal{L}_{2}(H,H_{n})}
      <& \infty ,
      \quad\text{for some}~
      \beta \in(0,1],
      \\
      \label{eq:dissipativity.corollary.n}
      \langle A_{n}\varphi + P_{n}F(\varphi),\varphi \rangle
      \leq&
      -c\|\varphi\|^{2} + C,
      \quad\forall\,
      \varphi \in \mathcal{D}(A_{n}),
      \\
      \label{eq:F.derivative.one.delta.n}
      \|(-A_{n})^{-\frac{\delta}{2}}P_{n}F'(\varphi)\psi\|
      \leq&
      L(1+\|\varphi\|_{1})\|\psi\|_{-1},
      \quad\forall\, \varphi \in \dot{H}^{1}, \psi \in H,\delta \in [1,2),
      \\
      \label{eq:F.derivative.second.n}
      \|(-A_{n})^{-\eta}P_{n}F''(\varphi)(\psi_{1},\psi_{2})\|
      \leq&
      L\|\psi_{1}\|\|\psi_{2}\|,\quad\forall\,
      \varphi,\psi_{1},\psi_{2} \in H,\eta \in[0,1),
      \\
      \label{eq:semigroup.smooth.one}
      \|(-A_{n})^{\gamma}E_{n}(t)\|_{\mathcal{L}(H_{n})}
      \leq&
      Ct^{-\gamma}e^{-\frac{\lambda_{1}}{2}t},
      \quad\forall\,
      t > 0,\gamma \geq 0,
      \\
      \label{eq:semigroup.smooth.two}
      \|(-A_{n})^{-\rho}(E_{n}(t) - E_{n}(s))\|_{\mathcal{L}(H_{n})}
      \leq&
      C(t-s)^{\rho}e^{-\frac{\lambda_{1}}{2}s},
      \quad\forall\,
      0 \leq s < t , \rho \in [0,1],
\end{align}
where $\beta,\delta,\eta$ are the same with the parameters in \eqref{eq:A.and.Q}, \eqref{eq:F.derivative.one.delta},
\eqref{eq:F.derivative.second}, respectively and the constants $c,C,L$ are independent of $n$ and $t$.
It is easy to verify the above estimates by taking the previous conditions or assertions into account.
For example, the proof of \eqref{eq:semigroup.smooth.one} is similar to that of the first assertion of \eqref{lem:semigroup.smooth} in Lemma \ref{lea:semigroupsmooth} as follows
\begin{equation*}
\begin{split}
\|(-A_{n})^{\gamma}&E_{n}(t)\|_{\mathcal{L}(H_{n})}
=
\sup\limits_{1 \leq i \leq n}\|(-A_{n})^{\gamma}E_{n}(t)e_{i}\|
=
\sup\limits_{1 \leq i \leq n}\lambda_{i}^{\gamma}e^{-\lambda_{i}t}
\\\leq&
2^\gamma t^{-\gamma}e^{-\frac{\lambda_{1}}{2}t}
\sup\limits_{1 \leq i \leq n}
\big(\tfrac{\lambda_{i}t}{2}\big)^{\gamma}e^{-\frac{\lambda_{i}t}{2}}
\leq
C t^{-\gamma}e^{-\frac{\lambda_{1}}{2}t},
\quad\forall\,t > 0,\gamma \geq 0.
\end{split}
\end{equation*}
Moreover, the above assumptions ensure that \eqref{eq:SDE} has a well-defined solution with a uniform mean square moment bound.

\begin{thm}[\textbf{Existence, uniqueness and moment boundedness of spatial approximation}]\label{th:exist.unique}
      Suppose that Assumptions \ref{ass:main.ass} and \ref{ass:main.ass2} hold. Then \eqref{eq:SDE} admits a unique solution $X^{n} \colon [0,\infty) \times \Omega \to H_{n}$ with continuous sample path given by
      \begin{equation}
            \label{eq:SDE.integral}
            X^{n}(t)
      =
            E_{n}(t)X_{0}^{n}
            +
            \int_{0}^{t}
                E_{n}(t-s)P_{n}F(X^{n}(s))
            \diff{s}
            +
            \int_{0}^{t}
                E_{n}(t-s)P_{n}
            \diff{W^{Q}(s)},
            \quad\forall\, t \geq 0,
            \quad\P\text{-a.s}.
      \end{equation}
      Moreover, there exists a constant
      $C = C(X_{0})>0$ independent of $n,t$ such that
      \begin{equation}
            \label{eq:moment.bound.n}
            \E\big[\|X^{n}(t)\|^{2}\big] \leq C.
      \end{equation}
\end{thm}

\begin{proof}
  It suffices to show \eqref{eq:moment.bound.n} since the existence of the unique solution $\{X^{n}(t)\}_{t\geq0}$ can be found in\cite[Theorem 4.5.3]{kloeden1992numerical}. In fact, set
  $
      \mathcal{O}^{n}(t)
      :=
      \int_{0}^{t}
          E_{n}(t-s)P_{n}
      \diff{W^{Q}(s)},
      \forall\,t \geq 0
  $,
  we can apply the It\^{o} isometry, \eqref{eq:A.and.Q.n} and \eqref{eq:semigroup.smooth.one} to derive that
  \begin{equation}
  \begin{split}
      \label{eq:eu.convolution}
      \E\big[\|\mathcal{O}^{n}(t)\|^{2}\big]
      =&
      \int_{0}^{t}
          \big\|
          E_{n}(t-s) P_{n} Q^{\frac{1}{2}}
          \big\|_{\mathcal{L}_{2}(H,H_{n})}^{2}
      \diff{s}
      \\\leq&
      \|
            (-A_{n})^{\frac{\beta-1}{2}}P_{n}Q^{\frac{1}{2}}
      \|_{\mathcal{L}_{2}(H,H_{n})}^{2}
      \int_{0}^{t}
          \|
               (-A_{n})^{\frac{1-\beta}{2}} E_{n}(t-s)
          \|_{\mathcal{L}(H_{n})}^{2}
      \diff{s}
      \\\leq&
      C\int_{0}^{t} (t-s)^{\beta-1} e^{-\lambda_{1}(t-s)} \diff{s}
      \leq C,
  \end{split}
  \end{equation}
  where in the last step we used the Gamma function
  \begin{equation}\label{eq:Gamma}
    \int_0^\infty x^{\varrho-1}e^{-x} \diff{x} < \infty, \quad\forall\, \varrho>0.
  \end{equation}
  Define $\bar{X}^{n}(t) := X^{n}(t) - \mathcal{O}^{n}(t), \forall\,t \geq 0$, then it satisfies the following partial differential equation
  \begin{equation}
      \label{eq:PDE}
            \frac{\diff{\bar{X}^{n}(t)}}{\diff{t}}
            =
            A_{n}\bar{X}^{n}(t)
            +
            P_{n}F(\bar{X}^{n}(t)+\mathcal{O}^{n}(t)),
            \quad\forall\, t > 0,
            \quad
            \bar{X}^{n}(0)=X_{0}^{n}.
  \end{equation}
  As a result, we have
  \begin{equation}
  \begin{split}
      \frac{\diff{e^{ct}\|\bar{X}^{n}(t)\|^{2}}}{\diff{t}}
      =
      2e^{ct}
      \big\langle
      A_{n}\bar{X}^{n}(t) + P_{n}F(\bar{X}^{n}(t)+\mathcal{O}^{n}(t))
      ,
      \bar{X}^{n}(t)
      \big\rangle
      +
      ce^{ct}\|\bar{X}^{n}(t)\|^{2},
  \end{split}
  \end{equation}
  where the constant $c$ comes from \eqref{eq:dissipativity.corollary.n}. Employing \eqref{eq:dissipativity.corollary.n}, \eqref{eq:F.derivative.one}, the Cauchy--Schwarz inequality and the weighted Young inequality $ab \leq \varepsilon{a^{2}} + \frac{b^{2}}{4\varepsilon}$ for all $a,b \in \R$ with $\varepsilon = \frac{c}{2L} > 0$ leads to
  \begin{equation}
  \begin{split}
      \label{eq:eu.inte}
      e^{ct}\|\bar{X}^{n}(t)\|^{2}
  =&
      \|\bar{X}_{0}^{n}\|^{2}
      +
      2\int_{0}^{t}
      e^{cs}
      \big\langle
      A_{n}\bar{X}^{n}(s) + P_{n}F(\bar{X}^{n}(s))
      ,
      \bar{X}^{n}(s)
      \big\rangle
      \diff{s}
      +
      c\int_{0}^{t} e^{cs}\|\bar{X}^{n}(s)\|^{2} \diff{s}
      \\&+
      2\int_{0}^{t}
      e^{cs}
      \big\langle
      P_{n}F(\bar{X}^{n}(s)+\mathcal{O}^{n}(s))
      -
      P_{n}F(\bar{X}^{n}(s))
      ,
      \bar{X}^{n}(s)
      \big\rangle
      \diff{s}
\\\leq&
      \|\bar{X}_{0}^{n}\|^{2}
      +
      2\int_{0}^{t}
      e^{cs} (-c\|\bar{X}^{n}(s)\|^{2} + C)
      \diff{s}
      +
      c\int_{0}^{t} e^{cs}\|\bar{X}^{n}(s)\|^{2} \diff{s}
      \\&+
      2\int_{0}^{t}
      e^{cs}
      \|
             P_{n}F(\bar{X}^{n}(s)+\mathcal{O}^{n}(s))
             -
             P_{n}F(\bar{X}^{n}(s))
       \|
      \|\bar{X}^{n}(s)\|
      \diff{s}
  \\\leq&
      \|\bar{X}_{0}^{n}\|^{2}
      -
      c\int_{0}^{t} e^{cs} \|\bar{X}^{n}(s)\|^{2} \diff{s}
      +
      2C\frac{e^{ct}-1}{c}
      +
      2L\int_{0}^{t}
                e^{cs} \| \mathcal{O}^{n}(s)\| \|\bar{X}^{n}(s)\|
            \diff{s}
  \\\leq&
      \|\bar{X}_{0}^{n}\|^{2}
      +
      2C\frac{e^{ct}-1}{c}
      +
      \frac{L^{2}}{c}
      \int_{0}^{t}
          e^{cs} \|\mathcal{O}^{n}(s)\|^{2}
      \diff{s}.
  \end{split}
  \end{equation}
  Taking expectations on the both sides of \eqref{eq:eu.inte} and using \eqref{eq:eu.convolution} yield
  \begin{equation}\label{eq:eu.last}
  \begin{split}
      e^{ct}\E\big[\|\bar{X}^{n}(t)\|^{2}\big]
  \leq&
      \|\bar{X}_{0}^{n}\|^{2}
      +
      2C\frac{e^{ct}-1}{c}
      +
      \frac{CL^{2}}{c}
      \int_{0}^{t}
          e^{cs}
      \diff{s}
  \leq
      \|\bar{X}_{0}^{n}\|^{2}
      + Ce^{ct}.
  \end{split}
  \end{equation}
Multiplying $e^{-ct}$ on the both sides of \eqref{eq:eu.last} gives
$\E\big[\|\bar{X}^{n}(t)\|^{2}\big]
\leq \|\bar{X}_{0}^{n}\|^{2} + C$. This together with \eqref{eq:eu.convolution} and $\E\big[\|X^{n}(t)\|^{2}\big]
\leq 2\E\big[\|\bar{X}^{n}(t)\|^{2}\big] + 2\E\big[\|\mathcal{O}^{n}(t)\|^{2}\big]$ results in the required conclusion \eqref{eq:moment.bound.n}.
\end{proof}

\subsection{Ergodicity for the spatial discretization}
           \label{sec:ergodicity.semi}

In order to give a sufficient condition for the ergodicity of the semi-discretization approximations process $\{X^{n}(t)\}_{t\geq0}$, we begin with a definition.

\begin{den}[Lyapunov condition]
	  Denote the solution $\{X^{n}(t)\}_{t\geq0}$ of \eqref{eq:SDE} with initial
	  value $X^{n}(0) = z \in H_{n}$ by $\{X^{n}(t,z)\}_{t\geq0}$ and
      let $V \colon H_{n} \to [0,\infty]$ be a Borel function with compact level sets
	  $K_{a} := \{ x \in H_{n} : V(x) \leq a \}$
	  for all $a > 0$.
	  We say that $\{X^{n}(t)\}_{t\geq0}$ satisfies the Lyapunov condition
	  if there exist $z \in H_{n}$ and $C(z) > 0$ such that
	  \begin{equation}\label{eq:lyapunov.initial}
	  \E[V(X^{n}(t,z))] \leq C(z), \quad \forall\, t > 0.
	  \end{equation}
\end{den}

Now we have the following theorem, see Proposition 7.10, Theorems 7.6 and 5.16 in \cite{da2006introduction}.
\begin{thm}\label{th:semi.ergodic}
      If the solution $\{X^{n}(t)\}_{t\geq0}$ of \eqref{eq:SDE} satisfies the Lyapunov condition,
      then $\{X^{n}(t)\}_{t\geq0}$ possesses at least one invariant probability measure. If in addition it happens that the corresponding Markov semigroup of $\{X^{n}(t)\}_{t\geq0}$ is strong Feller and irreducible, then $\{X^{n}(t)\}_{t\geq0}$ possesses a unique invariant probability measure and hence is ergodic.
\end{thm}

With the above theorem, we have the following result.

\begin{thm}[\textbf{Ergodicity of $\{ X^{n}(t) \}_{t \geq 0}$}]\label{th:ergodic.two}
      Suppose that Assumptions \ref{ass:main.ass} and \ref{ass:main.ass2} hold. Then $\{ X^{n}(t) \}_{t \geq 0}$ given by \eqref{eq:SDE.integral} is ergodic with a unique invariant measure $\nu^{n}$.
\end{thm}

\begin{proof}
       To prove the ergodicity of $\{ X^{n}(t) \}_{t \geq 0}$, let us first give an equivalent form of \eqref{eq:SDE}. Since $\{ X^{n}(t) \}_{t \geq 0}$
       is an $H_{n}$-valued stochastic
       process, we have
       \begin{equation}
             \label{eq:sf.zero}
             X^{n}(t)
             =
             \sum\limits_{i=1}^{n} x_{i}(t) e_{i},
             \quad
             x_{i}(t)
             =
             \langle X^{n}(t) , e_{i} \rangle ,
             \quad\forall\,
             i = 1, 2, \ldots, n.
      \end{equation}
      Inserting \eqref{eq:SDE} with \eqref{eq:W.Q.t} into
      $
             x_{i}(t)
      =
             \langle X^{n}(t) , e_{i} \rangle
      $
      yields
      \begin{equation}
            \label{eq:sf.one}
            \diff{x_{i}(t)}
      =
            \big(
                -\lambda_{i}x_{i}(t)
                +
                \langle P_{n}F(X^{n}(t)),e_{i} \rangle
            \big)
            \diff{t}
      +
            \sqrt{q_{i}}  \diff{\beta_{i}(t)},
            \quad\forall\,
             t > 0, i = 1, 2, \ldots, n.
      \end{equation}
From now on, we use $B'$ to denote the transpose of a vector or matrix $B$. By denoting
      \begin{align*}
           x(t)
      =
           ( x_{1}(t), x_{2}(t),& \ldots, x_{n}(t) )'
           \in
           \R^{n},
      \quad
           \beta(t)
      =
           ( \beta_{1}(t), \beta_{2}(t), \ldots, \beta_{n}(t) )'
           \in
           \R^{n},
      \\
           \Lambda
           =
           \text{diag}(-\lambda_{1},&\ldots,-\lambda_{n}) \in \R^{n \times n},
           \quad
           \bar{Q}
           =
           \text{diag}(\sqrt{q_{1}},\ldots,\sqrt{q_{n}}) \in \R^{n \times n},
      \\
           g(x(t))
           =&
           \left(
                 \langle P_{n}F(X^{n}(t)) , e_{1} \rangle,
                 \ldots,
                 \langle P_{n}F(X^{n}(t)) , e_{n} \rangle
           \right)' \in \R^{n},
      \end{align*}
      we can rewrite \eqref{eq:sf.one} as an $\R^n$-valued SDE
      \begin{equation}
            \label{eq:sf.two}
            \diff{x(t)}
      =
            \big( \Lambda x(t) + g(x(t)) \big) \diff{t}
            +
            \bar{Q} \diff{\beta(t)},
            \quad\forall\, t > 0,
      \end{equation}
      and thus it suffices to show that $\{ x(t) \}_{t \geq 0}$ is ergodic. Indeed, the ergodicity of $\{ x(t) \}_{t \geq 0}$ implies there exists a random variable $\xi = (\xi_{1}, \xi_{2}, \ldots, \xi_{n})$ such that
      $
           \lim\limits_{t \to \infty}   x(t)
      =
           \xi
      $,
      i.e.,
      $
           \lim\limits_{t \to \infty}  x_{i}(t)
      =
           \xi_{i},~i = 1, 2, \ldots, n.
      $
      It follows that
      $
           \lim\limits_{t \to \infty}   X^{n}(t)
      =
           \sum_{i=1}^{n}   \xi_{i} e_{i},
      $
      which immediately ensures that $\{ X^{n}(t) \}_{t \geq 0}$ is
      ergodic by the definition of ergodicity.
      By Theorem \ref{th:semi.ergodic} the proof of the ergodicity of
      $\{x(t)\}_{t \geq 0}$ is equivalent to show that $\{x(t)\}_{t \geq 0}$ is strong Feller, irreducible and satisfies the Lyapunov condition \eqref{eq:lyapunov.initial}. In what follows we will validate these properties one by one. Thanks to $\text{Rank}(\bar{Q}) = n$, the strong Feller property of $\{ x(t)\}_{t \geq 0}$ follows immediately by
      \cite[Proposition 2.3.2]{cerrai2001second}.

      To show the irreducibility of $\{x(t)\}_{t \geq 0}$,
      we denote $G(x(t)) := \Lambda x(t) + g(x(t))$ in \eqref{eq:sf.two} to get
       \begin{equation}
             \label{eq:sf.three}
             \diff{x(t)}
       =
             G(x(t)) \diff{t}
             +
             \bar{Q} \diff{\beta(t)},
             \quad\forall\, t > 0.
       \end{equation}
       Let $y, y^{+} \in \R^{n}$, $\delta, t > 0$ be arbitrary
       and denote the solution of \eqref{eq:sf.three} with initial
       value $x(0) = y$ by $\{x(t,y)\}_{t \geq 0}$. By the definition of irreducibility,
       it suffices to prove that
       \begin{equation}\label{eq:sf.four}
             \P(| x(t,y) - y^{+} | < \delta) > 0,\quad\forall\, t > 0.
       \end{equation}
       Here and below, we denote $(\cdot,\cdot)$ to be the usual Euclidean inner product in $\R^n$ and $|\cdot|$ be the corresponding norm in $\R^n$, or the Frobenius matrix norm in $\R^{n \times n}$. To show \eqref{eq:sf.four}, we follow the idea stemmed from \cite{mattingly2002ergodicity} and consider the associated control problem
       \begin{equation}
             \label{eq:sf.five}
             \frac{\diff{\bar{x}(t)}}{\diff{t}}
       =
             G(\bar{x}(t))
             +
             \bar{Q} \frac{\diff{U(t)}}{\diff{t}},
             \quad\forall\, t > 0.
       \end{equation}
       Then for every fixed $t > 0$, we can find a control function
       $U \in C^{1}([0,t];\R^{n})$ with $U(0) = 0$ such
       that \eqref{eq:sf.five} is
       satisfied and $\bar{x}(0) = y, \bar{x}(t) = y^{+}$.
       This can be achieved by polynomial interpolation between the end
       points using a linear polynomial in time with vector coefficients
       in $\R^{n}$ and by the invertibility of matrix $\bar{Q}$.
       The integral forms of \eqref{eq:sf.three}
       and \eqref{eq:sf.five} show that
       \begin{equation*}
             x(s,y) - \bar{x}(s)
       =
             \int_{0}^{s} G(x(r,y)) - G(\bar{x}(r)) \diff{r}
             +
             \bar{Q} (\beta(s) - U(s)),
             \quad
             \forall\, s \in [0,t].
       \end{equation*}
       Note that the event
       $
           \Omega_{t}^{\varepsilon}
           :=
           \{
               \omega \in \Omega
               :
               \sup\limits_{0 \leq s \leq t}
               | \beta(s)(\omega) - U(s) |
           \leq
               \varepsilon
           \}
       $
       occurs with positive probability for any $\varepsilon > 0$ and $t > 0$,
       since the Wiener measure of any such tube is positive
       (see \cite[Lemma 3.4]{mattingly2002ergodicity}).
       Observing that $G$ is global Lipschitz continuous
       because of \eqref{eq:F.derivative.one}, one sees that
       \begin{equation*}
             | x(s,y)(\omega) - \bar{x}(s) |
       \leq
             L_{G} \int_{0}^{s} | x(r,y)(\omega) - \bar{x}(r) | \diff{r}
             +
             \varepsilon |\bar{Q}|,
             \quad \forall\, s \in [0,t], \omega \in \Omega_{t}^{\varepsilon}.
       \end{equation*}
       By Gronwall's inequality, 
       we have
       $
             | x(s,y)(\omega) - \bar{x}(s) |
       \leq
             \varepsilon |\bar{Q}| e^{sL_{G}}
       $ for all $s \in [0,t]$ and $\omega \in \Omega_{t}^{\varepsilon}$.
       Choosing $s = t$ and
       $\varepsilon = \delta / (| \bar{Q} | e^{tL_{G}} )$ and observing $\bar{x}(t) = y^{+}$,
       \eqref{eq:sf.four} holds and the irreducibility follows.

       Now we are in a position to verify that $\{x(t)\}_{t \geq 0}$ satisfies the Lyapunov condition \eqref{eq:lyapunov.initial}. For this, we choose a Borel function $V(x) = | x |^{2}, x \in \R^{n}$.
       Because of the continuity of norm and the Heine--Borel theorem in the finite-dimensional space $\R^{n}$, it follows that the level sets $K_{a}$ are compact for all $a > 0$. By \eqref{eq:sf.two} and the It\^{o} formula, we have
       \begin{equation}\label{eq:Lc.one}
       \begin{split}
             \diff{| x(t) |^{2}}
       =&
             2 \big( x(t) , \Lambda x(t) \big) \diff{t}
             +
             2 \big( x(t) , g(x(t)) \big) \diff{t}
             +
             \sum\limits_{i=1}^{n} q_{i} \diff{t}
             +
             2 \big( x(t) , \bar{Q} \dif{\beta(t)} \big)
       \\\leq&
             -2 \lambda_{1} |x(t)|^{2} \diff{t}
             +
             2 \big( x(t) , g(x(t)) \big) \diff{t}
             +
             \sum\limits_{i=1}^{n} q_{i} \diff{t}
             +
             2 \big( x(t) , \bar{Q} \dif{\beta(t)} \big).
       \end{split}
       \end{equation}
       Recall the notation $x(t)$ and $g(x(t))$, we use the projection property of $P_n$ and \eqref{eq:disipativity} to get
       \begin{equation*}
       \begin{split}
             \big( x(t) , g(x(t)) \big)
       =&
             \sum\limits_{i=1}^{n}
             \big\langle X^{n}(t) , e_{i} \big\rangle
             \big\langle P_{n} F(X^{n}(t)) , e_{i} \big\rangle
       =
             \big\langle
                 X^{n}(t) , P_{n} F(X^{n}(t))
             \big\rangle
       \\=&
             \big\langle
                 X^{n}(t) , F(X^{n}(t))-F(0)
             \big\rangle
             +
             \big\langle X^{n}(t) , F(0) \big\rangle
       \\\leq&
             L_{F}  \|X^{n}(t)\|^{2}
             +
             \tfrac{\lambda_{1}-L_{F}}{2} \|X^{n}(t)\|^{2}
             +
             \tfrac{\|F(0)\|^{2}}{2(\lambda_{1}-L_{F})}
       \\=&
             \tfrac{\lambda_{1}+L_{F}}{2} \|X^{n}(t)\|^{2}
             +
             \tfrac{\|F(0)\|^{2}}{2(\lambda_{1}-L_{F})},
       \end{split}
       \end{equation*}
       where we used the weighted Young inequality $ab \leq \varepsilon{a^{2}} + \frac{b^{2}}{4\varepsilon}$ for all $a,b \in \R$ with $\varepsilon = \frac{\lambda_{1}-L_{F}}{2} > 0$.
       Observing
       $
             \|X^{n}(t) \|^{2}
       =
             |x(t)|^{2}
       $
       because of \eqref{eq:sf.zero} and taking expectations
       on the both sides of \eqref{eq:Lc.one} show that
       \begin{equation*}
             \frac{\diff{\E \big[|x(t)|^{2}\big]}}
                  {\diff{t}}
       \leq
             -(\lambda_{1}-L_{F}) \E \big[|x(t)|^{2}\big]
             +
             \Big(
                   \tfrac{\|F(0)\|^{2}}{\lambda_{1}-L_{F}}
                   +
                   \sum\limits_{i=1}^{n} q_{i}
             \Big),
       \end{equation*}
       which leads to
       \begin{equation*}
       \begin{split}
             \E \big[ | x(t) |^{2} \big]
       \leq&
             e^{-(\lambda_{1}-L_{F})t}\E\big[|x(0)|^{2}\big]
             +
             \tfrac{1-e^{-(\lambda_{1}-L_{F})t}}{\lambda_{1}-L_{F}}
             \Big(
                   \tfrac{\|F(0)\|^{2}}{\lambda_{1}-L_{F}}
                   +
                   \sum\limits_{i=1}^{n} q_{i}
             \Big)
       \\\leq&
             \E\big[|x(0)|^{2}\big]
             +
             \tfrac{1}{\lambda_{1}-L_{F}}
             \Big(
                   \tfrac{\|F(0)\|^{2}}{\lambda_{1}-L_{F}}
                   +
                   \sum\limits_{i=1}^{n} q_{i}
             \Big).
       \end{split}
       \end{equation*}
       This means that $\{x(t)\}_{t \geq 0}$ satisfies the Lyapunov condition \eqref{eq:lyapunov.initial} and thus finishes the proof.
\end{proof}

\subsection{
Weak spatial approximation error over long time
}
           \label{eq:weak.error.semi}

An important ingredient to obtain the time-independent weak error is the improved estimates on the derivatives of the solution of the associated Kolmogorov equation.
To show this, for any $n \in \N$ and $\Phi \in C_{b}^{2}(H,\R)$ we introduce the function $v^{n} \colon [0,\infty) \times H_{n} \to \R$ by
\begin{equation}
      \label{eq:kol.solu}
      v^{n}(t,y)
      =
      \E \big[\Phi(X^{n}(t,y))\big],
      \quad\forall\,
      t \geq 0,y \in H_{n},
\end{equation}
where $X^{n}(t,y)$ is the unique solution of \eqref{eq:SDE} with the initial value $ X_{0}^{n} = y$. Recall that $v^{n}(t,y)$ is continuously differentiable with respect to $t$ and continuously twice differentiable with respect to $y$ and acts as the unique strict solution of the following Kolmogorov equation
\begin{equation}\label{eq:kol}
\left\{
\begin{array}{l}
      \hspace{-0.5em}
             \frac{\partial{v^{n}(t,y)}}{\partial{t}}
         =
             \big\langle
             Dv^{n}(t,y)
             ,
             A_{n}y+P_{n}F(y)
             \big\rangle
             +
             \frac{1}{2}
             \tr\big\{
                       D^{2}v^{n}(t,y)
                       (P_{n}Q^{\frac{1}{2}})
                       (P_{n}Q^{\frac{1}{2}})^{*}
                \big\},\quad\forall\, t > 0,
      \\
      \hspace{-0.5em}
             v^{n}(0,y)
         =
             \Phi(y),
\end{array}
\right.
\end{equation}
under Assumptions~\ref{ass:main.ass} and \ref{ass:main.ass2}, see \cite[Theorem 9.16]{da1992stochastic}.
Here by a strict solution of \eqref{eq:kol} we mean a function
$v^{n} \in C_{b}^{1,2}([0,\infty) \times H_{n}, \R)$ such that \eqref{eq:kol} holds.
Moreover, by the Riesz representation theorem, we can always identify the first derivative $Dv^{n}(t,y)$ at $y \in H_{n}$
with an element in $H_{n}$
and the second derivative $D^{2}v^{n}(t,y)$ at $y \in H_{n}$ with a bounded linear operator on $H_{n}$.

Repeating the proof of Propositions 5.1 and 5.2 in
\cite{brehier2014approximation} with slight changes and taking Assumptions \ref{ass:main.ass} and \ref{ass:main.ass2} into account,
we have the following regularity results on the derivatives of $v^{n}(t,y)$.

\begin{prn}[\textbf{Regularity of $Dv^{n}(t,y)$ and $D^{2}v^{n}(t,y)$}]
      \label{pro:regularityDvntx}
      Suppose that Assumptions \ref{ass:main.ass} and \ref{ass:main.ass2} hold. Let $v^{n}(t,y)$ be defined by \eqref{eq:kol.solu} with $\Phi \in C_{b}^{2}(H,\R)$. Then for any $\gamma \in [0,1)$ and $\gamma_{1},\gamma_{2} \in [0,1)$ with $\gamma_{1}+\gamma_{2} < 1$ there exist $C_{\gamma}, C_{\gamma_{1},\gamma_{2}},\tilde{c}>0$ such that
      \begin{align}\label{eq:kol.solu.FIRST}
            \|(-A_{n})^{\gamma}Dv^{n}(t,y) \|
      \leq&
            C_{\gamma}(1+t^{-\gamma})e^{-\tilde{c}t},
      \\
            \label{eq:kol.solu.SECOND}
            \|
              (-A_{n})^{\gamma_{2}}
              D^{2}v^{n}(t,y)
              (-A_{n})^{\gamma_{1}}
            \|_{\mathcal{L}(H_{n})}
      \leq&
            C_{\gamma_{1},\gamma_{2}}
            (1+t^{-\eta}+t^{-(\gamma_{1}+\gamma_{2})})e^{-\tilde{c}t}
      \end{align}
      for all $t \geq 0$ and $y \in H_n$, where the parameter $\eta$ comes from \eqref{eq:F.derivative.second}.
\end{prn}

With the above preparations, we can prove the following time-independent weak error.

\begin{thm}
      [\textbf{Spatial weak error}]
      \label{th:spatial}
      Suppose that Assumptions \ref{ass:main.ass} and \ref{ass:main.ass2} hold. Let $\{X(t)\}_{t\geq0}$ and $\{X^{n}(t)\}_{t\geq0}$
      be given by \eqref{eq:SPDE} and \eqref{eq:SDE}, respectively. Then for any $T > 0$, $n \in \N$
      and $\Phi \in C_{b}^{2}(H,\R)$ there exists $C>0$ independent of $T,n$ such that
      \begin{equation}\label{eq:saptial.order}
         \big|\E[\Phi(X(T))] - \E[\Phi(X^{n}(T))]\big|
         \leq
         C\lambda_{n}^{-\beta+\epsilon}.
      \end{equation}
\end{thm}

\begin{proof}
      We set $k \in \N  \cap [n,\infty)$ and decompose the spatial
      approximation error as follows
      \begin{equation}
      \begin{split}
            \label{eq:saptial.decompose}
            \big|
                   \E[ \Phi(X(T)) ]
                   -
                   \E[ \Phi(X^{n}(T)) ]
            \big|
      \leq&
            \big|
                   \E[ \Phi(X(T)) ]
                   -
                   \E[ \Phi(X^{k}(T)) ]
            \big|
      +
            \big|
                   \E[ \Phi(X^{k}(T)) ]
                   -
                   \E[ \Phi(X^{n}(T)) ]
            \big|.
      \end{split}
      \end{equation}
      Taking $k \to \infty$ in~\eqref{eq:saptial.decompose} and
      employing the fact that $X^{k}(T)$ converges to $X(T)$ in
      mean square sense (see,~e.g.,~\cite[Lemma~A.1]{wang2013weak}) lead to
      \begin{equation}
            \label{eq:saptial.limsup}
            \big|
                    \E[ \Phi(X(T)) ]
                    -
                    \E[ \Phi(X^{n}(T)) ]
            \big|
      \leq
            \limsup\limits_{k \to \infty}
            \big|
                    \E[ \Phi(X^{k}(T)) ]
                    -
                    \E[ \Phi(X^{n}(T)) ]
            \big|.
      \end{equation}
      By \eqref{eq:kol.solu} and \eqref{eq:kol}, it follows that
      \begin{equation}\label{eq:spatial.limsup.decompose}
      \begin{split}
            \E[\Phi(X^{k}(T))] - \E[\Phi(X^{n}(T))]
      =
            \E[v^{k}(T,X^{k}_{0})]
            -
            \E[ v^{k}(T,X^{n}_{0})]
      +
            \E[ v^{k}(T,X^{n}_{0})]
            -
            \E[ v^{k}(0,X^{n}(T))].
      \end{split}
      \end{equation}
      Before we calculate the first term on the right hand
      side of \eqref{eq:spatial.limsup.decompose}, we
      note that
      \begin{equation}
      \begin{split}
            \label{eq:spatial.limsup.inequality}
            \| (P_{k} - P_{n}) v \|
      \leq&
            \| P_{k} \|_{\mathcal{L}(H_{k})}
            \| (I - P_{n}) (-A)^{-\beta} \|_{\mathcal{L}(H_{k})}
            \| (-A)^{\beta}v \|
      \leq
             \lambda_{n}^{-\beta} \| v \|_{2\beta},
             \quad\forall\,v \in \dot{H}^{2\beta}.
      \end{split}
      \end{equation}
      We then use Taylor's formula, \eqref{eq:kol.solu.FIRST},
      \eqref{eq:spatial.limsup.inequality} and $X_{0} \in \dot{H}^{2\beta}$ to obtain
      \begin{equation}
      \begin{split}\label{eq:spatial.limsup.first}
            \big|
                \E[ v^{k}(T,X^{k}_{0}) ]
                -&
                \E[ v^{k}(T,X^{n}_{0}) ]
             \big|
      =
            \Big|
            \int_{0}^{1}
                \E\big[
                             \big\langle
                             Dv^{k}
                             (
                                  T,X^{n}_{0}
                                  +
                                  r(X^{k}_{0}-X^{n}_{0})
                             )
                             ,
                             X^{k}_{0}-X^{n}_{0}
                             \big\rangle
                  \big]
            \diff{r}
            \Big|
      \\\leq&
            \int_{0}^{1}
            \E\big[
              \|
                  Dv^{k}
                  \big(
                          T,X^{n}_{0}
                          +
                          r(X^{k}_{0}-X^{n}_{0})
                  \big)
              \|
              \|(P_{k} - P_{n}) X_{0} \|
              \big]
            \diff{r}
      \\\leq&
            C\lambda_{n}^{-\beta} e^{-\tilde{c}T} \| X_{0} \|_{2\beta}
        \leq
            C\lambda_{n}^{-\beta}.
      \end{split}
      \end{equation}
      Now we process to consider the second term on the right hand
      side of \eqref{eq:spatial.limsup.decompose}. Applying the It\^{o}
      formula to $v^{k}(T-t,X^{n}(t)),\forall\,t \in [ 0, T]$, one sees that
       \begin{equation}\label{eq:saptial.limsup.second.ito}
       \begin{split}
             \E[v^{k}(0,X^{n}(T))]
              &-
              \E[v^{k}(T,X^{n}_{0})]
       =
                \int_{0}^{T}
                    \E\bigg[-
                    \frac{\partial{v^{k}(T-t,X^{n}(t))}}
                         {\partial{t}}\bigg]
                \diff{t}
          \\&+
                \int_{0}^{T}
                \E\big[
                    \big\langle
                    Dv^{k}(T-t,X^{n}(t))
                    ,
                    A_{n} X^{n}(t) + P_{n}F(X^{n}(t))
                    \big\rangle
                \big]
                \diff{t}
          \\&+
              \frac{1}{2}
                \int_{0}^{T}
                \E\big[
                    \tr\big\{
                               D^{2}v^{k}(T-t,X^{n}(t))
                               (P_{n}Q^{\frac{1}{2}})
                               (P_{n}Q^{\frac{1}{2}})^{*}
                       \big\}
                \big]
                \diff{t}.
       \end{split}
       \end{equation}
       Substituting \eqref{eq:kol} into \eqref{eq:saptial.limsup.second.ito} and using $A_{n} X^{n}(t) - A_{k}X^{n}(t) = 0$ for $k \in \N \cap [n,\infty)$ enable us to get
       \begin{equation}\label{eq:I.one.two.three}
       \begin{split}
             \E[v^{k}(0,&X^{n}(T))]
              -
              \E[v^{k}(T,X^{n}_{0})]
       =
                \int_{0}^{T}
                \E\big[
                  \big\langle
                    Dv^{k}(T-t,X^{n}(t))
                    ,
                    ( P_{n} - P_{k} )  F(X^{n}(t))
                    \big\rangle
                \big]
                \diff{t}
          \\&+
              \frac{1}{2}
              \int_{0}^{T}
                  \E\big[
                    \tr\big\{
                               D^{2}v^{k}(T-t,X^{n}(t))
                               (P_{n}-P_{k})Q^{\frac{1}{2}}
                               \big(P_{n}Q^{\frac{1}{2}}\big)^{*}
                       \big\}
                \big]
                \diff{t}
          \\&+
              \frac{1}{2}
                \int_{0}^{T}
                \E\big[
                    \tr\big\{
                               D^{2}v^{k}(T-t,X^{n}(t))
                               \big(P_{k}Q^{\frac{1}{2}}\big)
                               \big((P_{n}-P_{k})Q^{\frac{1}{2}}\big)^{*}
                       \big\}
                \big]
                \diff{t}
       :=I_{1}+I_{2}+I_{3}.
       \end{split}
       \end{equation}
       In the sequel we will estimate $I_1, I_2, I_3$ separately. By \eqref{eq:kol.solu.FIRST}, 
       \eqref{eq:F.derivative.one}, \eqref{eq:moment.bound.n} and \eqref{eq:Gamma}, we have
       \begin{equation}\label{eq:I.one}
       \begin{split}
             |I_{1}|
       \leq&
              \int_{0}^{T}
                  \E\big[
                    \|(-A_{k})^{1-\epsilon}
                       Dv^{k}(T-t,X^{n}(t))
                    \|
                    \|
                       (-A_{k})^{-1+\epsilon}
                       \big(  ( P_{n} - P_{k} )  F(X^{n}(t)) \big)
                    \|
                  \big]
              \diff{t}
       \\\leq&
              C\lambda_{n}^{-1+\epsilon}
              \int_{0}^{T}
                  \big(1 + (T-t)^{-1+\epsilon}\big)
                  e^{-\tilde{c}(T-t)}
                  \E\big[ \|F(X^{n}(t))\| \big]
              \diff{t}
       \leq
              C\lambda_{n}^{-1+\epsilon}.
       \end{split}
       \end{equation}
       Here we emphasize that the error constant $C$ in the last term of \eqref{eq:I.one} is independent of time $T$ and thus \eqref{eq:I.one} is essentially different from the analogue estimation (69) in \cite{wang2016weak}, which allows the error constant $C$ to depend on $T$. Such estimation may cause explosion as time $T$ goes to infinity and is no longer working for \eqref{eq:344}. So here and below we sharp similar estimations appeared in \cite{wang2016weak} via some new arguments and techniques to adapt our purpose.
       Concerning $I_{2}$,
       we can derive from \eqref{eq:.operator.Gamma.one} and \eqref{eq:.operator.Gamma.three} that
       \begin{equation*}
       \begin{split}
             |I_{2}|
       =&
             \frac{1}{2}
             \Big|
               \int_{0}^{T}
               \E\big[
                   \tr\big\{
                               (-A_{k})^{\frac{1-\beta}{2}}
                               D^{2}v^{k}(T-t,X^{n}(t))
                               (P_{n}-P_{k})Q^{\frac{1}{2}}
                               \big((-A_{k})^{\frac{\beta-1}{2}}
                                    P_{n}Q^{\frac{1}{2}}\big)^{*}
                      \big\}
               \big]
               \diff{t}
             \Big|
       \\\leq&
               \frac{1}{2}
               \int_{0}^{T}
               \E\big[
                   \big\|
                             (-A_{k})^{\frac{1-\beta}{2}}
                             D^{2}v^{k}(T-t,X^{n}(t))
                             (-A_{k})^{\frac{1+\beta}{2}-\epsilon}
                   \big\|_{\mathcal{L}(H_{k})}
                 \big]
               \diff{t}
                   \\&~~~~~~~~~\cdot
                   \big\|
                             (-A_{k})^{-\frac{1+\beta}{2}+\epsilon}
                             (P_{n}-P_{k})
                             (-A_{k})^{\frac{1-\beta}{2}}
                   \big\|_{\mathcal{L}(H_{k})}
                   \\&~~~~~~~~~\cdot
                   \big\|
                             (-A_{k})^{\frac{\beta-1}{2}}P_{k}
                             Q^{\frac{1}{2}}
                             \big((-A_{k})^{\frac{\beta-1}{2}}
                                   P_{n}Q^{\frac{1}{2}}\big)^{*}
                   \big\|_{\mathcal{L}_{1}(H_{k})}.
       \end{split}
       \end{equation*}
       Noticing that
       $
               \big\|
                          (-A_{k})^{\frac{\beta-1}{2}}
                          P_{n}Q^{\frac{1}{2}}
               \big\|_{\mathcal{L}_{2}(H,H_{k})}
               =
               \big\|
                          (-A_{n})^{\frac{\beta-1}{2}}
                          P_{n}Q^{\frac{1}{2}}
               \big\|_{\mathcal{L}_{2}(H,H_{n})}
       $
       and applying
       \eqref{eq:kol.solu.SECOND},
       \eqref{eq:.operator.Gamma.two},
       \eqref{eq:A.and.Q.n} and \eqref{eq:Gamma}
       bring about
       \begin{equation}
             \label{eq:I.two}
       \begin{split}
             |I_{2}|
       \leq&
               C
               \int_{0}^{T}
                   \big(1+(T-t)^{-\eta}
                          +(T-t)^{-(1-\epsilon)}\big)
                               e^{-\tilde{c}(T-t)}
               \diff{t}
               \big\|
                         (-A_{k})^{-\beta+\epsilon}
                         (P_{n}-P_{k})
               \big\|_{\mathcal{L}(H_{k})}
                \\&\cdot
               \big\|
                          (-A_{k})^{\frac{\beta-1}{2}}
                          P_{k}Q^{\frac{1}{2}}
               \big\|_{\mathcal{L}_{2}(H,H_{k})}
               \big\|
                          (-A_{n})^{\frac{\beta-1}{2}}
                          P_{n}Q^{\frac{1}{2}}
               \big\|_{\mathcal{L}_{2}(H,H_{n})}
       \leq
               C\lambda_{n}^{-\beta+\epsilon}.
       \end{split}
       \end{equation}
       Similarly to $I_{2}$, we can arrive at
       \begin{equation}
             \label{eq:I.three}
             |I_{3}|
       \leq
             C\lambda_{n}^{-\beta+\epsilon}.
       \end{equation}
       Inserting \eqref{eq:I.one}, \eqref{eq:I.two}
       and \eqref{eq:I.three} into \eqref{eq:I.one.two.three} gives
       \begin{equation*}\label{eq:I.onetwothree}
             \big|
             \E[v^{k}(0,X^{n}(T))]
             -
             \E[v^{k}(T,X^{n}_{0})]
             \big|
             \leq
             C\lambda_{n}^{-\beta+\epsilon}.
       \end{equation*}
       This together with \eqref{eq:spatial.limsup.first} and \eqref{eq:spatial.limsup.decompose} verifies the desired result \eqref{eq:saptial.order}.
\end{proof}

\subsection{Error of invariant measures for spatial discretization}
            \label{sec:order.semi}


\begin{thm}
      \label{th:order.measure.spatial}
       Suppose that Assumptions \ref{ass:main.ass} and \ref{ass:main.ass2} hold. Let $\nu$ and $\nu^{n}$ be the corresponding unique invariant measures of $\{X(t)\}_{t \geq 0}$ and $\{ X^{n}(t) \}_{t \geq 0}$, respectively. Then for any $T > 0,n \in \N$ and $\Phi \in C_{b}^{2}(H,\R)$ there exists $C>0$ independent of $T,n$ such that
      \begin{equation}
            \label{eq:order.measure.spatial}
            \Big|
            \int_{H} \Phi(y) \,\nu(\dif{y})
            -
            \int_{H_{n}} \Phi(y) \,\nu^{n}(\dif{y})
            \Big|
            \leq
            C\lambda_{n}^{-\beta+\epsilon}.
      \end{equation}
\end{thm}

\begin{proof}
      From Theorems \ref{th:EueSPDEs} and \ref{th:ergodic.two}, we know $\{X(t)\}_{t\geq0}$ and $\{ X^{n}(t) \}_{t \geq 0}$ are ergodic. This together with the definition of ergodicity implies \eqref{eq:ergodic.one} and
      \begin{align}
            \label{eq:ergodic.two}
            \lim\limits_{T\to\infty}
            \frac{1}{T}
            \int_{0}^{T} \E\big[\Phi(X^{n}(t))\big] \diff{t}
            =&
            \int_{H_{n}} \Phi(y) \,\nu^{n}(\dif{y})
            \quad\text{in}~L^{2}(H,\nu),
            \quad
            \forall\,\Phi \in C_{b}^{2}(H,\R),
      \end{align}
       and hence
      \begin{equation}\label{eq:344}
      \begin{split}
            &
            \Big|
            \int_{H} \Phi(y) \,\nu(\dif{y})
            -
            \int_{H_{n}} \Phi(y) \,\nu^{n}(\dif{y})
            \Big|
      \leq
            \lim\limits_{T\to\infty}
            \frac{1}{T}
            \int_{0}^{T}
            \big|
            \E\big[\Phi(X(t))\big]
            -
            \E\big[\Phi(X^{n}(t))\big]
            \big|
            \diff{t}
      \leq
            C\lambda_{n}^{-\beta+\epsilon},
      \end{split}
      \end{equation}
      where \eqref{eq:saptial.order}  was used in the last step.
\end{proof}

\begin{rek}\label{re:I}
Note that two important classes of noise are included here. One is the space-time white noise in the case $Q = I$ and the other is the trace class
noise in the case $\tr(Q) < \infty$. For the space-time white noise, it is well-known that \eqref{eq:A.and.Q} is fulfilled with $\beta < \frac{1}{2}$ in space dimension $d = 1$ \cite[Remark 3.2]{kovacs2010finite}. In this situation our result indicates that the convergence order between $\nu$ and $\nu^n$ is $1-\epsilon$ for arbitrarily small $\epsilon>0$. For the trace class noise, \eqref{eq:A.and.Q} is satisfied with $\beta = 1$ \cite[Remark 3.2]{kovacs2010finite} and our result implies that the convergence order between $\nu$ and $\nu^n$ is $2-\epsilon$ for arbitrarily small $\epsilon>0$,
in space dimension $d = 1$.
\end{rek}

\section{Spatio-temporal full discretization and its ergodicity}

We will apply an exponential Euler scheme to \eqref{eq:SDE} to obtain a spatio-temporal full discretization approximation
$\{Y_{m}^{n}\}_{m \in \N}$ and give some regularity estimates in Subsection \ref{sec:discrete.full}.
Subsection \ref{sec:ergodicity.full} shows that $\{Y_{m}^{n}\}_{m \in \N}$
is ergodic with a unique invariant measure $\nu_{\tau}^{n}$ via the theory of geometric ergodicity of Markov chain.
Based on a weak error representation formula, the time-independent weak error is investigated in Subsection \ref{sec:weak.error.full}.
Armed with the ergodicity and weak error estimate, we finally obtain the error between invariant measures $\nu^{n}$ and $\nu_{\tau}^{n}$ in Subsection \ref{sec:order.full}.

Throughout this section, we need the following notation.
Let $\tau > 0 $ be the uniform time stepsize. Further let $m,M \in \N$ and set $t_{m} = m\tau$ and $T =M\tau$. Moreover,
the generic constant $C$ must be independent of the spatial dimension $n$ and the final time $T =M\tau$
but may depend on $X_{0}$, $\Phi$, $L_{F}$, $L$ and other parameters.

\subsection{Exponential Euler scheme}
           \label{sec:discrete.full}

Now we approximate \eqref{eq:SDE} in time by the exponential Euler scheme
\begin{equation}
      \label{eq:EES}
      Y_{m}^{n}
      =
      E_{n}(\tau)Y_{m-1}^{n}
      +
      {\tau}E_{n}(\tau)P_{n}F(Y_{m-1}^{n})
      +
      E_{n}(\tau)P_{n}\Delta{W_{m-1}^{Q}},
      \quad
      Y_{0}^{n} = X_{0}^{n},
\end{equation}
where $Y_{m}^{n}$ is an approximation of $X^{n}(t_{m})$ and
$
      E_{n}(\tau)P_{n}\Delta{W_{m-1}^{Q}}
  :=
      \int_{t_{m-1}}^{t_{m}}
         E_{n}(\tau)P_{n}
      \diff{W^{Q}(s)}
$
is well defined since
$ E_{n}(\tau)P_{n}Q^{\frac{1}{2}} \colon H \to H_{n} $
is a Hilbert--Schmidt operator.

The following lemma concerns the regularity of $\{Y_{m}^{n}\}_{m \in \N}$ over long time.

\begin{lea}\label{lem:bound.momnet}
      Suppose that Assumptions \ref{ass:main.ass} and \ref{ass:main.ass2} hold. Let $\tau < \tau_0 \leq \frac{\lambda_1-L_F}{4L^2}$
      and let $\{Y_{m}^{n}\}_{m \in \N}$ be given by \eqref{eq:EES}. Then for any
      $n, m \in \N$ and $\gamma \in [0,\frac{\beta}{2})$, there exists $C > 0$ independent of $n,m$ such that
      \begin{equation}\label{eq:bound.momnet}
            \E\big[\|(-A_{n})^{\gamma}Y_{m}^{n}\|^{2}\big]
            \leq
            C.
      \end{equation}
\end{lea}

\begin{proof}
      We first prove the following inequalities
      \begin{equation}\label{eq:bound.momnet.two}
            \E\big[\|Y_{m}^{n}\|^{2}\big] \leq C,
            \quad
            \E\big[\|F(Y_{m}^{n})\|^{2}\big] \leq C.
      \end{equation}
      Indeed, it suffices to verify the first inequality of \eqref{eq:bound.momnet.two} since the second one is an immediate consequence of the first one and \eqref{eq:F.derivative.one}. Now we can easily rewrite \eqref{eq:EES} as
      \begin{equation}\label{eq:Y.m.n.rewrite}
      \begin{split}
            Y_{m}^{n}
      =&
      E_{n}^{m}(\tau)Y_{0}^{n}
      +
      {\tau}\sum\limits_{i=0}^{m-1}E_{n}^{m-i}(\tau)
      P_{n}F(Y_{i}^{n})
      +
      \sum\limits_{i=0}^{m-1}E_{n}^{m-i}(\tau)
      P_{n}\Delta{W_{i}^{Q}}.
      \end{split}
      \end{equation}
      Set $\lfloor {s} \rfloor = t_{i}$ for $s \in [t_{i},t_{i+1})$, $i=0,1,\ldots,m-1$ and denote
      \begin{equation*}
            \mathcal{O}_{m}^{n}
      :=
            \sum\limits_{i=0}^{m-1}E_{n}^{m-i}(\tau)
            P_{n}\Delta{W_{i}^{Q}}
      =
            \int_{0}^{t_{m}}
                E_{n}(t_{m}-\lfloor {s} \rfloor)P_{n}
            \diff{W^{Q}(s)},
      \end{equation*}
      then by the It\^{o} isometry, \eqref{eq:semigroup.smooth.one},
      \eqref{eq:A.and.Q.n} and \eqref{eq:Gamma} we have
      \begin{equation}\label{eq:mathcal.O.m.n}
      \begin{split}
            \E\big[\|\mathcal{O}_{m}^{n}\|^{2}\big]
      =&
            \int_{0}^{t_{m}}
                \big\|
                    E_{n}(t_{m}-\lfloor {s} \rfloor)P_{n}Q^{\frac{1}{2}}
                \big\|_{\mathcal{L}_{2}(H,H_{n})}^{2}
            \diff{s}
      \\\leq&
            \int_{0}^{t_{m}}
                \big\|
                    (-A_{n})^{\frac{1-\beta}{2}}
                    E_{n}(t_{m}-\lfloor {s} \rfloor)
                \big\|_{\mathcal{L}(H_{n})}^{2}
                \big\|
                    (-A_{n})^{\frac{\beta-1}{2}}P_{n}Q^{\frac{1}{2}}
                \big\|_{\mathcal{L}_{2}(H,H_{n})}^{2}
            \diff{s}
      \\\leq&
            C\int_{0}^{t_{m}}
                 (t_{m}-\lfloor {s} \rfloor)^{\beta-1}
                 e^{-\lambda_{1}(t_{m}-\lfloor {s} \rfloor)}
            \diff{s}
        \leq    C.
      \end{split}
      \end{equation}
      This together with \eqref{eq:F.derivative.one} indicates
      \begin{equation}\label{eq:PnF.mathcal.O.m.n}
      \begin{split}
            \E\big[\|P_{n}F(\mathcal{O}_{m}^{n})\|^{2}\big]
      \leq&
            2L^{2}
            \E\big[\|\mathcal{O}_{m}^{n}\|^{2}\big]
            +
            2\|F(0)\|^{2}
      \leq
            C.
      \end{split}
      \end{equation}
      Set
            $\bar{Y}_{m}^{n} := Y_{m}^{n} - \mathcal{O}_{m}^{n}$,
      it is obvious that $\bar{Y}_{0}^{n} = Y_{0}^{n}$ and
      \begin{equation*}
      \begin{split}
            \bar{Y}_{m}^{n}
      =&
            E_{n}^{m}(\tau)\bar{Y}_{0}^{n}
            +
            {\tau}\sum\limits_{i=0}^{m-1}E_{n}^{m-i}(\tau)
            P_{n}F(\bar{Y}_{i}^{n}+\mathcal{O}_{i}^{n}),
      \end{split}
      \end{equation*}
      which immediately gives
      \begin{equation*}
      \begin{split}
            \bar{Y}_{m}^{n}
      =&
            E_{n}(\tau)\bar{Y}_{m-1}^{n}
            +
            {\tau}E_{n}(\tau)
            P_{n}F(\bar{Y}_{m-1}^{n}+\mathcal{O}_{m-1}^{n}).
      \end{split}
      \end{equation*}
       According to $\| E_{n}(\tau) \|_{\mathcal{L}(H_{n})} \leq e^{-\lambda_{1}\tau}$
       and \eqref{eq:disipativity}, \eqref{eq:F.derivative.one}, we have
      \begin{equation*}
      \begin{split}
            \|\bar{Y}_{m}^{n}\|^{2}
      \leq&
            \|E_{n}(\tau)\|_{\mathcal{L}(H_{n})}^{2}
            \big(
            \|\bar{Y}_{m-1}^{n}\|^{2}
            +
            \tau^2\|P_{n}F(\bar{Y}_{m-1}^{n}+\mathcal{O}_{m-1}^{n})\|^{2}
            +
            2\tau\langle
                  \bar{Y}_{m-1}^{n}
                  ,
                  P_{n}F(\bar{Y}_{m-1}^{n}+\mathcal{O}_{m-1}^{n})
               \rangle
            \big)
      \\\leq&
            e^{-2\lambda_{1}\tau}
            \big(
            \|\bar{Y}_{m-1}^{n}\|^{2}
            +
            2\tau^2\|P_{n}F(\bar{Y}_{m-1}^{n}+\mathcal{O}_{m-1}^{n})
                   -P_{n}F(\mathcal{O}_{m-1}^{n})\|^{2}
            +
            2\tau^2\|P_{n}F(\mathcal{O}_{m-1}^{n})\|^{2}
            \\&+
            2\tau\langle
                  \bar{Y}_{m-1}^{n}
                  ,
                  P_{n}F(\bar{Y}_{m-1}^{n}+\mathcal{O}_{m-1}^{n})
                  -
                  P_{n}F(\mathcal{O}_{m-1}^{n})
               \rangle
            +
            2\tau\langle
                  \bar{Y}_{m-1}^{n}
                  ,
                  P_{n}F(\mathcal{O}_{m-1}^{n})
               \rangle
            \big)
      \\\leq&
            \big( 1 + 2{\tau}L_{F} + 2{\tau}^{2}L^{2}
            + \tfrac{\lambda_{1} - L_{F}}{2}{\tau} \big)
            e^{-2\lambda_{1}\tau} \|\bar{Y}_{m-1}^{n}\|^{2}
            +
            2\big({\tau}^{2}+\tfrac{\tau}{\lambda_{1} - L_{F}}\big)e^{-2\lambda_{1}\tau}
            \|P_{n}F(\mathcal{O}_{m-1}^{n})\|^{2},
      \end{split}
      \end{equation*}
      where we used the weighted Young inequality
      $ab \leq \varepsilon{a^{2}} + \frac{b^{2}}{4\varepsilon}$
      for all $a,b \in \R$ with $\varepsilon = \frac{\lambda_{1} - L_{F}}{4}> 0$. Observing $\tau < \tau_0 \leq \frac{\lambda_1-L_F}{4L^2}$, we have $2{\tau}^{2}L^{2} \leq \frac{\lambda_{1} - L_{F}}{2}{\tau}$
and consequently
$$
1 + 2{\tau}L_{F} + 2{\tau}^{2}L^{2}
+ \tfrac{\lambda_{1} - L_{F}}{2}{\tau}
\leq
1 + (L_{F}+\lambda_{1}){\tau} 
\leq
e^{ (\lambda_{1}+L_{F})\tau}
$$
due to the inequality $1+x \leq e^{x}$ for all $x \in \R$. Then
$e^{-(\lambda_{1}+L_{F})\tau} \leq \max\{1,e^{-(\lambda_{1}+L_{F})\tau_{0}}\}$
and \eqref{eq:PnF.mathcal.O.m.n} result in
      \begin{equation*}
      \begin{split}
            \E\big[\|\bar{Y}_{m}^{n}\|^{2}\big]
      \leq&
            e^{-(\lambda_{1}-L_{F}){\tau}}
            \E\big[\|\bar{Y}_{m-1}^{n} \|^{2}\big]
            +
            Ce^{-(\lambda_{1}-L_{F}){\tau}}{\tau}
      \\\leq&
            e^{-(\lambda_{1}-L_{F})m{\tau}}
            \E\big[\|\bar{Y}_{0}^{n}\|^{2}\big]
            +
            \frac{Ce^{-(\lambda_{1}-L_{F}){\tau}}{\tau}}
                 {1-e^{-(\lambda_{1}-L_{F}){\tau}}}
      \\\leq&
            \|X_{0}\|^{2}
            +
            \tfrac{C}{\lambda_{1}-L_{F}},
      \end{split}
      \end{equation*}
      which yields the first inequality of \eqref{eq:bound.momnet.two} because of
      \eqref{eq:mathcal.O.m.n} and $Y_{m}^{n}=\bar{Y}_{m}^{n} + \mathcal{O}_{m}^{n}$.
      With regard to \eqref{eq:bound.momnet}, we derive from \eqref{eq:Y.m.n.rewrite} that
      \begin{equation*}
      \begin{split}
            Y_{m}^{n}
      =&
            E_{n}(t_{m})Y_{0}^{n}
            +
            \int_{0}^{t_{m}}
                E_{n}(t_{m}-\lfloor {s} \rfloor)
                P_{n}F(Y_{\lfloor {s/\tau} \rfloor}^{n})
            \diff{s}
            +
            \int_{0}^{t_{m}}
                E_{n}(t_{m}-\lfloor {s} \rfloor)P_{n}
            \diff{W^{Q}(s)}.
      \end{split}
      \end{equation*}
      Using the It\^{o} isometry, \eqref{eq:semigroup.smooth.one}, \eqref{eq:bound.momnet.two}, \eqref{eq:A.and.Q} and $X_{0} \in \dot{H}^{\beta}$ leads to
      \begin{equation*}
      \begin{split}
            \|(-A_{n})^{\gamma}&Y_{m}^{n}\|_{L^2(\Omega,H_n)}
      \leq
            \|
                 (-A_{n})^{\gamma}
                 E_{n}(t_{m})Y_{0}^{n}
            \|_{L^2(\Omega,H_n)}
            +
               \Big\|
                   \int_{0}^{t_{m}}
                       (-A_{n})^{\gamma}
                       E_{n}(t_{m}-\lfloor {s} \rfloor)P_{n}
                   \diff{W^{Q}(s)}
               \Big\|_{L^2(\Omega,H_n)}
            \\&+
            \int_{0}^{t_{m}}
                   \big\|
                       (-A_{n})^{\gamma}
                       E_{n}(t_{m}-\lfloor {s} \rfloor)
                       P_{n}F(Y_{\lfloor {s/\tau} \rfloor}^{n})
                   \big\|_{L^2(\Omega,H_n)}
            \diff{s}
      \\\leq&
            \|E_{n}(t_{m})\|_{\mathcal{L}(H_{n})}
             \|(-A_{n})^{\gamma}Y_{0}^{n}\|
            +
            \int_{0}^{t_{m}}
                  \|
                      (-A_{n})^{\gamma}
                      E_{n}(t_{m}-\lfloor {s} \rfloor)
                  \|_{\mathcal{L}(H_{n})}
                  \|
                      P_{n}F(Y_{\lfloor {s/\tau} \rfloor}^{n})
                  \|_{L^2(\Omega,H_n)}
            \diff{s}
            \\&+
            \Big(
            \int_{0}^{t_{m}}
                   \|
                       (-A_{n})^{\gamma-\frac{\beta-1}{2}}
                       E_{n}(t_{m}-\lfloor {s} \rfloor)
                   \|_{\mathcal{L}(H_{n})}^{2}
                   \|
                       (-A_{n})^{\frac{\beta-1}{2}}
                       P_{n}Q^{\frac{1}{2}}
                   \|_{\mathcal{L}_{2}(H,H_{n})}^{2}
             \diff{s}
             \Big)^{\frac{1}{2}}
      \\\leq&
            C
            +
            C\int_{0}^{t_{m}}
                  (t_{m}-\lfloor {s} \rfloor)^{-\gamma}
                  e^{-\frac{\lambda_1}{2}(t_{m}-\lfloor {s} \rfloor)}
            \diff{s}
            +
            \Big(
            C\int_{0}^{t_{m}}
                (t_{m}-\lfloor {s} \rfloor)^{-2\gamma+\beta-1}
                e^{-\lambda_1(t_{m}-\lfloor {s} \rfloor)}
             \diff{s}
             \Big)^{\frac{1}{2}}.
      \end{split}
      \end{equation*}
      Observing $1-\gamma>0,-2\gamma+\beta>0$, we finally use \eqref{eq:Gamma} to obtain \eqref{eq:bound.momnet} and thus complete the proof.
\end{proof}

Furthermore, we can show the following result.

\begin{lea}
      \label{lem:bound.momnet.1/2}
      Suppose that Assumptions \ref{ass:main.ass} and \ref{ass:main.ass2} hold.
      Let $\tau < \tau_0 \leq \frac{\lambda_1-L_F}{4L^2}$ and let $\{Y_{m}^{n}\}_{m \in \N}$ be given by \eqref{eq:EES}.
      Then for any $n,m \in \N$ and arbitrarily small $\epsilon > 0$ there exists $C >0$ independent of $n,m$ such that
      \begin{equation}\label{eq:bound.momnet.1/2}
            \E\big[\|(-A_{n})^{\frac{1}{2}}Y_{m}^{n}\|^{2}\big]
            \leq
            C{\tau}^{\beta-\epsilon-1}.
      \end{equation}
\end{lea}

\begin{proof}
      Making use of \eqref{eq:EES}, H\"{o}lder's inequality and It\^{o}'s isometry gives
      \begin{equation*}
      \begin{split}
            \E\big[\|(-A_{n})^{\frac{1}{2}}Y_{m}^{n}\|^{2}\big]
      \leq&
            3\E\big[
               \|
               (-A_{n})^{\frac{1}{2}}E_{n}(\tau)Y_{m-1}^{n}
               \|^{2}
               \big]
            +
            3\tau^{2}\E\big[
                       \|
                       (-A_{n})^{\frac{1}{2}}E_{n}(\tau)
                       P_{n}F(Y_{m-1}^{n})
                       \|^{2}
                       \big]
            \\&+
            3\E\big[
               \|
               (-A_{n})^{\frac{1}{2}}E_{n}(\tau)
               P_{n}\Delta{W_{m-1}^{Q}}
               \|^{2}
               \big]
      \\\leq&
            3\|
             (-A_{n})^{\frac{1-(\beta-\varepsilon)}{2}}
             E_{n}(\tau)
             \|_{\mathcal{L}(H_{n})}^{2}
            \E\big[
               \|
               (-A_{n})^{\frac{\beta-\varepsilon}{2}}
               Y_{m-1}^{n}
               \|^{2}
               \big]
            \\&+
            3\tau^{2}
            \| (-A_{n})^{\frac{1}{2}}E_{n}(\tau) \|_{\mathcal{L}(H_{n})}^{2}
            \E\big[ \|P_{n}F(Y_{m-1}^{n}) \|^{2} \big]
            \\&+
            3\tau
             \|
             (-A_{n})^{\frac{1-(\beta-1)}{2}}E_{n}(\tau)
             \|_{\mathcal{L}(H_{n})}^{2}
             \|
             (-A_{n})^{\frac{\beta-1}{2}}
             P_{n}Q^{\frac{1}{2}}
             \|_{\mathcal{L}_{2}(H,H_{n})}^{2}
      \\\leq&
            C{\tau}^{\beta-\epsilon-1}
            +
            C\tau
            +
            C\tau^{\beta-1}
           =
            C{\tau}^{\beta-\epsilon-1}(1+{\tau}^{-(\beta-\epsilon)+2}+\tau^{\epsilon}),
      \end{split}
      \end{equation*}
      where we also applied \eqref{eq:semigroup.smooth.one}, \eqref{eq:bound.momnet}--\eqref{eq:bound.momnet.two} and \eqref{eq:A.and.Q.n} in the penultimate step. The fact that $\tau \in (0,\tau_{0})$, $\epsilon > 0$ and  $-(\beta-\epsilon)+2 > 0$ finally ends the proof.
\end{proof}

\subsection{Ergodicity for the space-time full discretization}
           \label{sec:ergodicity.full}

To prove the ergodicity of $\{Y_{m}^{n}\}_{m \in \N}$, we introduce the theory of geometric ergodicity of Markov chain, which was first established by Mattingly, Stuart and Higham in \cite{mattingly2002ergodicity} to prove ergodicity of several discretizations based on backward Euler method for SDEs. Then it was applied in \cite{chen2017approximation} to test ergodicity of a modified implicit Euler method for an ergodic one-dimensional damped stochastic nonlinear Schr\"{o}dinger equation.

\begin{asn}[Lyapunov condition]\label{ass:Lyapunov.condition}
      There is a function $V \colon \R^{d} \to [1,\infty)$ with $\lim\limits_{|x| \to \infty} V(x) = \infty$ and real numbers $\alpha_{1} \in (0,1)$, $\alpha_{2} \in [0,\infty)$ such that
      $$
                \E\big[ V(x_{k+1}) | \F_{k} \big]
          \leq
                \alpha_{1} V(x_{k}) + \alpha_{2},
      $$
      where $\F_{k}$ denotes the $\sigma$-algebra of events up to and including the $k$-th iteration.
\end{asn}

\begin{den}
      We say that $V$ is essentially quadratic if there exist
      $C_{i}>0, i = 1, 2, 3$, such that
      $$
                 C_{1} (1 + |x|^{2})
          \leq
                 V(x)
          \leq
                 C_{2} (1 + |x|^{2}),
          \quad
                 |\nabla{V(x)}| \leq C_{3}(1 + |x|),
          \quad\, \forall\, x \in \R^d.
      $$
\end{den}


\begin{asn}[Minorization condition]\label{ass:minorization.condition}
  The Markov chain $\{ x_{k} \}_{k \in \N}$ on a state space
$( \R^{d}, \mathcal{B}(\R^{d}) )$ with transition kernel
$
      P_{k}(x,B)
  :=
      \P(x_{k} \in B | x_{0} = x) ,
      k \in \N, x \in \R^{d},
      B \in \mathcal{B}(\R^{d})
$
satisfies, for some fixed compact set $S \in \mathcal{B}(\R^{d})$,
\begin{enumerate}[(i)]
      \item for some $y^{*} \in \text{int}(S)$ there is, for any
            $\delta > 0$, a $\bar{k} = \bar{k}(\delta) \in \N$
            such that
            $$
                 P_{\bar{k}}(y, B_{\delta}(y^{*}))
                 >
                 0,
                 ~~~\forall\, y \in S,
            $$
            where $B_{\delta}(y^{*})$ denotes the open ball of radius $\delta$ centered at $y^{*}$;
      \item for $k \in \N$ the transition kernel $ P_{k}(x,B)$ possesses a density
            $p_{k}(x,y)$ such that
            $$
                  P_{k}(x,B)
              =
                  \int_{B} p_{k}(x,y) \diff{y},
              ~~~\forall\,
                  x \in S,
                  B \in \mathcal{B}(\R^{d}) \cap \mathcal{B}(S)
            $$
            and $p_{k}(x,y)$ is jointly continuous in
            $(x,y) \in S \times S$.
\end{enumerate}
\end{asn}

The following theorem comes from Theorem 2.5 in \cite{mattingly2002ergodicity}.

\begin{thm}
      \label{th:full.ergodic}
      If Markov chain $\{ x_{k} \}_{k \in \N}$ satisfies Assumptions \ref{ass:Lyapunov.condition} and \ref{ass:minorization.condition}  with an essentially quadratic Lyapunov function $V$, then $\{ x_{k} \}_{k \in \N}$ is ergodic with a unique invariant measure.
\end{thm}

Armed with the above theorem, we can prove the following result.

\begin{thm}[\textbf{Ergodicity of $\{ Y_{m}^{n} \}_{m \in \N}$}]\label{th:ergodic.three}
      Suppose that Assumptions \ref{ass:main.ass} and \ref{ass:main.ass2} hold and let $\tau < \tau_0 \leq \frac{\lambda_1-L_F}{4L^2}$. Then $\{ Y_{m}^{n} \}_{m \in \N}$ given by \eqref{eq:EES} is ergodic with a unique invariant measure $\nu_{\tau}^{n}$.
\end{thm}

\begin{proof}
      In view of \eqref{eq:W.Q.t}, we can rewrite \eqref{eq:EES} as
      \begin{equation}\label{eq:414}
            Y_{m}^{n}
            =
            E_{n}(\tau)Y_{m-1}^{n}
            +
            {\tau}E_{n}(\tau)P_{n}F(Y_{m-1}^{n})
            +
            \sum\limits_{i=1}^{n}
            \sqrt{q_{i}}e^{-\lambda_{i}\tau}
            \Delta\beta_{i}^{m-1}e_{i}
      \end{equation}
      with the Wiener increments
      $
          \Delta\beta_{i}^{m-1}
          :=
          \beta_{i}(t_{m}) - \beta_{i}(t_{m-1}), i = 1,2,\ldots,n, m \in \N.
      $
      Owing to the independence of $\{\Delta\beta_{i}^{m}\}_{i=1}^{n}, m \in \N$, it follows from \cite[Page xix]{borovkov1998ergodicity} that the random variables make a Markov chain.
      According to Theorem \ref{th:full.ergodic}, it suffices to show $\{ Y_{m}^{n} \}_{m \in \N}$ satisfies the Lyapunov condition and the minorization condition.

      Let us first show the Lyapunov condition.
      Choosing $V(x) = \| x \|^{2} + 1, x \in H_{n}$, it is easy to verify that $V$ is essentially quadratic. From \eqref{eq:EES} and the properties of conditional
      expectation, we have
      \begin{equation}\label{eq:estimation}
      \begin{split}
            \E\big[ V(Y_{m+1}^{n}) | \F_{m} \big]
      =&
            \big\| E_{n}(\tau)\big( Y_{m}^{n} + {\tau}P_{n}F(Y_{m}^{n}) \big) \big\|^{2}
            +
            \E\big[\| E_{n}(\tau)P_{n}\Delta{W_{m}^{Q}}\|^{2}\big]  +  1.
      \end{split}
      \end{equation}
      Observing
      $
            \|E_{n}(\tau)\|_{\mathcal{L}({H}_{n})}
            \leq
            e^{-\lambda_{1}{\tau}}
      $, \eqref{eq:disipativity}--\eqref{eq:F.derivative.one} and applying the weighted Young inequality $ab \leq \varepsilon{a^{2}} + \frac{b^{2}}{4\varepsilon}$
      for all $a,b \in \R$ with $\varepsilon = 3L^2\tau > 0$ enable us to show that
      \begin{equation}\label{eq:bigEntaubigYmn}
      \begin{split}
            &\big\| E_{n}(\tau)\big( Y_{m}^{n} + {\tau}P_{n}F(Y_{m}^{n}) \big) \big\|^{2}
        \leq
             e^{-2\lambda_{1}{\tau}}
             \big( \|Y_{m}^{n}\|^{2}
             +
             {\tau}^{2} \|F(Y_{m}^{n})\|^{2}
             +
             2{\tau}\langle Y_{m}^{n},F(Y_{m}^{n}) \rangle
             \big)
      \\\leq&
             e^{-2\lambda_{1}{\tau}}
             \big( \|Y_{m}^{n}\|^{2}
             +
             2{\tau}^{2} \|F(Y_{m}^{n})-F(0)\|^{2}
             +
             2{\tau}^{2} \|F(0)\|^{2}
             +
             2{\tau}\langle Y_{m}^{n},F(Y_{m}^{n})-F(0) \rangle
             +
             2{\tau}\langle Y_{m}^{n},F(0) \rangle
             \big)
      \\\leq&
             e^{-2\lambda_{1}{\tau}}
             \big( \|Y_{m}^{n}\|^{2}
             +
             8L^{2}{\tau}^{2} \|Y_{m}^{n}\|^{2}
             +
             2{\tau}^{2} \|F(0)\|^{2}
             +
             2{\tau}L_F\|Y_{m}^{n}\|^{2}
             +
             \|F(0)\|^{2}/(6L^2)
             \big)
      \\\leq&
             e^{-2\lambda_{1}{\tau}}
             \big( 1 + 2\lambda_1{\tau} \big)\|Y_{m}^{n}\|^{2}
             +
             e^{-2\lambda_{1}{\tau}}\|F(0)\|^{2}
             (1+12{\tau}^{2}L^2)/(6L^2),
      \end{split}
      \end{equation}
      where we used $\tau < \tau_0 \leq \frac{\lambda_1-L_F}{4L^2}$ in the last step. Employing It\^o's isometry, \eqref{eq:semigroup.smooth.one}
      and \eqref{eq:A.and.Q.n} implies
      \begin{equation}
            \label{eq:estimation.pre}
      \begin{split}
            &
            \E\big[
              \|
                  E_{n}({\tau})
                  P_{n}\Delta{W_{m}^{Q}}
              \|^{2}
              \big]
            =
            \tau
            \big\|
                E_{n}({\tau})P_{n}
                Q^{\frac{1}{2}}
            \big\|_{\mathcal{L}_{2}(H,H_{n})}^{2}
         \\\leq&
            \tau
            \big\|
                (-A_{n})^{\frac{1-\beta}{2}}
                E_{n}({\tau})
            \big\|_{\mathcal{L}(H_{n})}^{2}
            \big\|
                (-A_{n})^{\frac{\beta-1}{2}}
                P_{n}Q^{\frac{1}{2}}
            \big\|_{\mathcal{L}_{2}(H,H_{n})}^{2}
         \leq
            C\tau^{\beta}e^{-\lambda_{1}{\tau}}.
      \end{split}
      \end{equation}
      Inserting \eqref{eq:bigEntaubigYmn} and \eqref{eq:estimation.pre} into \eqref{eq:estimation}, one can derive
      \begin{equation*}
      \begin{split}
            \E\big[ V(Y_{m+1}^{n}) | \F_{m} \big]
      \leq&
            \alpha_{1}\|Y_{m}^{n}\|^{2}
            +
            \alpha_{2}
      \end{split}
      \end{equation*}
      with $ \alpha_{1}
          :=
                \big( 1 + 2\lambda_{1}{\tau} \big)
                e^{-2\lambda_{1}{\tau}}
          \in   (0,1)$ and
      \begin{align*}
                \alpha_{2}
          :=&
                e^{-2\lambda_{1}{\tau}}\|F(0)\|^{2}(1+12{\tau}^{2}L^2)/(6L^2)
                +
                C\tau^{\beta}e^{-\lambda_{1}{\tau}} +  1
          \in   [0,\infty),
      \end{align*}
      which says that Assumption \ref{ass:Lyapunov.condition} is fulfilled with an essentially quadratic Lyapunov function $V$.

      Now we are ready to prove the minorization condition. By the Heine--Borel theorem in the finite-dimensional space $H_n$, we know that $S_{n}:=\{s \in H_{n}: \|s\| \leq 1\} \in \mathcal{B}(H_n)$ is a compact set.
      For any $s \in S_{n}$ and $z  \in Z_{n}$ with $Z_{n} \in \mathcal{B}(H_n)$, we use \eqref{eq:414} together with $\{e_{i}\}_{i=1}^{n}$ being an orthonormal basis of $H_{n}$ to get
      \begin{equation*}
            \Delta\beta_{i}^{m-1}
            =
            \big( e^{\lambda_{i}\tau}\langle s,e_{i} \rangle
            -
            \langle z,e_{i} \rangle
            -
            {\tau}\langle P_{n}F(z),e_{i} \rangle \big) / \sqrt{q_{i}},
            \quad\forall\, i = 1,2, \ldots, n,
      \end{equation*}
      which shows that $\{\Delta\beta_{i}^{m-1}\}_{i=1}^{n}$ can be properly chosen to guarantee that $Y_{m}^n = s$ starting from $Y_{m-1}^n = z$. Then the first condition in Assumption \ref{ass:minorization.condition} is fulfilled thanks to the property that Brownian motions hit any cylindrical set with positive probability.
      It remains to show the second condition in Assumption \ref{ass:minorization.condition}. Since each Gaussian random variable $\Delta\beta_{i}^{m-1}$ admits  $C^{\infty}$ density function, and so does $E_{n}(\tau)P_{n}\Delta{W_{m-1}^{Q}}$ in  \eqref{eq:414}, then the transition kernel $P_{1}(x,B_{n})$ with $ x \in S_{n} \in \mathcal{B}(H_{n}), B_{n} \in \mathcal{B}(H_{n}) \cap \mathcal{B}(S_{n})$ possesses a density $p_{1}(x,y)$, which is jointly continuous in $(x,y) \in S_{n} \times S_{n}$. Finally, the time-homogeneous property of Markov chain $\{ Y_{m}^{n} \}_{m \in \N}$ promises the joint continuity of densities $p_{m}(x,y), m \in \N$. Thus we complete the proof by Theorem \ref{th:full.ergodic}. 
\end{proof}

\subsection{Weak temporal approximation error over long time}\label{sec:weak.error.full}

Armed with our assumptions, one can easily check that all conditions of the weak error representation formula introduced in \cite[Theorem~2.2]{wang2016weak} are fulfilled. Therefore, we can apply this formula to carry out an easy weak error analysis via some elementary arguments. To adapt our analysis, the formula is listed below with some non-essential changes.

\begin{thm}[\textbf{Weak error representation formula}]
      Suppose that Assumptions \ref{ass:main.ass} and \ref{ass:main.ass2} hold. Then for any $T = M\tau$ and $\Phi \in C_{b}^{2}(H;\R)$ the weak error of the exponential Euler scheme \eqref{eq:EES} for the problem \eqref{eq:SDE} has the following representation
      \begin{equation}\label{eq:weak.error.representation}
      \begin{split}
      &
            \E\big[\Phi(X^{n}(T))\big]
            -
            \E\big[\Phi(Y_{M}^{n})\big]
      \\=&
            \sum\limits_{m=0}^{M-1}
                 \int_{t_{m}}^{t_{m+1}}
                     \E\big[
                                \big\langle
                                Dv^{n}(T-t,\tilde{Y}^{n}(t))
                                ,
                                P_{n}F(\tilde{Y}^{n}(t))
                                -
                                E_{n}(t-t_{m})P_{n}F(Y_{m}^{n})
                                \big\rangle
                       \big]
                 \diff{t}
            \\&+
            \frac{1}{2}
            \int_{t_{m}}^{t_{m+1}}
                \E\big[
                  \tr
                  \big\{
                      D^{2}v^{n}(T-t,\tilde{Y}^{n}(t))
                      \big(
                          (P_{n}Q^{\frac{1}{2}})(P_{n}Q^{\frac{1}{2}})^{*}
          \\&~~~-
                          (E_{n}(t-t_{m})P_{n}Q^{\frac{1}{2}})
                          (E_{n}(t-t_{m})P_{n}Q^{\frac{1}{2}})^{*}
                      \big)
                  \big\}
                  \big]
            \diff{t}.
      \end{split}
      \end{equation}
      Here $X^{n}(T)$ and $Y_{M}^{n}$ are determined
      by \eqref{eq:SDE.integral} and \eqref{eq:EES}, respectively,
      and $\tilde{Y}^{n}(t), \forall\,t \in [t_{m},t_{m+1}]$ is a
      continuous extension of $Y_{m}^{n}$, defined by
      \begin{equation}
            \label{eq:continuous.extension}
      \begin{split}
            \tilde{Y}^{n}(t)
      =&
            E_{n}(t-t_{m})
            \big(
                  Y_{m}^{n}
                  +
                  P_{n}F(Y_{m}^{n}) (t-t_{m})
                  +
                  P_{n}(W^{Q}(t)-W^{Q}(t_{m}))
            \big),\quad\forall\, t \in [t_{m},t_{m+1}],
      \end{split}
      \end{equation}
      where
      $
             E_{n}(t-t_{m})P_{n}\big(W^{Q}(t)-W^{Q}(t_{m})\big)
      :=
             \int_{t_{m}}^{t}
                 E_{n}(t-t_{m})P_{n}
             \diff{W^{Q}(s)}
      $.
\end{thm}

An approximation result between $\tilde{Y}^{n}(t)$ and $Y_{m}^{n}$ is given by the following lemma.

\begin{lea}
      Suppose that Assumptions \ref{ass:main.ass} and \ref{ass:main.ass2} hold. Let $\tau < \tau_0 \leq \frac{\lambda_1-L_F}{4L^2}$ and let $\{Y_{m}^{n}\}_{m \in \N}$ and $\tilde{Y}^{n}(t)$ be given by \eqref{eq:EES} and \eqref{eq:continuous.extension}, respectively.
      Then for any $n, m \in \N$, there exists $C>0$ independent of $n,m$ such that
      \begin{equation}\label{eq:continuous.extension.approximation}
            \E\big[
              \|
                  \tilde{Y}^{n}(t) - Y_{m}^{n}
              \|^{2}\big]
            \leq
            C{\tau}^{\beta-\epsilon},
            \quad
            \forall\,t \in [t_{m},t_{m+1}].
      \end{equation}
\end{lea}

\begin{proof}
      One can easily derive from \eqref{eq:continuous.extension} that
      \begin{equation}\label{eq:tildeYntYmn}
      \begin{split}
            \tilde{Y}^{n}(t) - Y_{m}^{n}
      =&
            (E_{n}(t-t_{m})-I)Y_{m}^{n}
            +
            \int_{t_{m}}^{t}
                E_{n}(t-t_{m})P_{n}F(Y_{m}^{n})
            \diff{s}
            +
            \int_{t_{m}}^{t}
                E_{n}(t-t_{m})P_{n}
            \diff{W^{Q}(s)}.
      \end{split}
      \end{equation}
      Using the inequality $|a+b+c|^2 \leq 3(|a|^2 + |b|^2 + |c|^2)$ for all $a,b,c \in \R$, we have
      \begin{equation*}
      \begin{split}
      \E\big[
      \| \tilde{Y}^{n}(t) - Y_{m}^{n} \|^{2}
      \big]
            \leq&
      3\E\big[
      \| (E_{n}(t-t_{m})-I)Y_{m}^{n} \|^{2}
      \big]
      +
      3\E\left[
      \Big\|
      \int_{t_{m}}^{t}
      E_{n}(t-t_{m})P_{n}F(Y_{m}^{n})
      \diff{s}
      \Big\|^{2}
      \right]
      \\&+
      3\E\left[
      \Big\|
      \int_{t_{m}}^{t}
      E_{n}(t-t_{m})P_{n}
      \diff{W^{Q}(s)}
      \Big\|^{2}
      \right].
      \end{split}
      \end{equation*}
      Then it follows from H\"older's inequality and It\^o's isometry that
      \begin{equation*}
      \begin{split}
            \E\big[
              \| \tilde{Y}^{n}(t) - Y_{m}^{n} \|^{2}
              \big]
      \leq&
            3\E\big[
               \| (E_{n}(t-t_{m})-I)Y_{m}^{n} \|^{2}
               \big]
            +
            3{\tau}^{2}
            \E\big[
              \| E_{n}(t-t_{m})P_{n}F(Y_{m}^{n}) \|^{2}
              \big]
            \\&+
            3(t-t_{m})
            \big\|
            E_{n}(t-t_{m}) P_{n}Q^{\frac{1}{2}}
            \big\|_{\mathcal{L}_{2}(H,H_{n})}^{2}
      \\\leq&
            3\big\|
                 (E_{n}(t-t_{m})-I)
                 (-A_{n})^{-\frac{\beta-\epsilon}{2}}
             \big\|_{\mathcal{L}(H_{n})}^{2}
            \E\big[
              \|
                  (-A_{n})^{\frac{\beta-\epsilon}{2}} Y_{m}^{n}
              \|^{2}
              \big]
            \\&+
            3{\tau}^{2} \| E_{n}(t-t_{m}) \|_{\mathcal{L}(H_{n})}^{2}
            \E\big[\| F(Y_{m}^{n})\|^{2}\big]
            \\&+
            3(t-t_{m})
             \big\|
                 E_{n}(t-t_{m})
                 (-A_{n})^{\frac{1-\beta}{2}}
             \big\|_{\mathcal{L}(H_{n})}^{2}
             \big\|
                 (-A_{n})^{\frac{\beta-1}{2}}
                 P_{n}Q^{\frac{1}{2}}
             \big\|_{\mathcal{L}_{2}(H,H_{n})}^{2}.
     \end{split}
     \end{equation*}
     One can further employ \eqref{eq:semigroup.smooth.one}--\eqref{eq:semigroup.smooth.two},
     \eqref{eq:bound.momnet}--\eqref{eq:bound.momnet.two}, \eqref{eq:A.and.Q.n} and the stability of the semigroup $\{E_{n}(t)\}_{t \geq 0}$ to get the desired result
     \eqref{eq:continuous.extension.approximation}.
\end{proof}

The next theorem gives a time-independent weak error.

\begin{thm}[\textbf{Temporal weak error}]\label{th:temporal.I}
      Suppose that Assumptions \ref{ass:main.ass} and \ref{ass:main.ass2} hold. Let $\tau < \tau_0 \leq \frac{\lambda_1-L_F}{4L^2}$ and
      let $\{X^{n}(t)\}_{t \geq 0}$ and $\{Y_{m}^{n}\}_{m \in \N}$ be given by \eqref{eq:SDE} and \eqref{eq:EES}, respectively.
      Then for any $T > 0$, $n,M \in \N$ and $\Phi \in C_{b}^{2}(H,\R)$ there exists $C>0$ independent of $T,n,M$ such that for any $T = M\tau$,
      \begin{equation}\label{eq:temporal.order.I}
            \big|
            \E[\Phi(X^{n}(T))]
            -
            \E[\Phi(Y_{M}^{n})]
            \big|
            \leq
            C{\tau}^{\beta-\epsilon}.
      \end{equation}
\end{thm}

\begin{proof}
      We first use \eqref{eq:weak.error.representation} to decompose the weak error at time $T = M{\tau}$ as follows
      \begin{equation}
            \label{eq:telescoping}
      \begin{split}
            &\E\big[\Phi(X^{n}(T))\big]
            -
            \E\big[\Phi(Y_{M}^{n})\big]
     \\=&
           \sum\limits_{m=0}^{M-1}
              \int_{t_{m}}^{t_{m+1}}
                  \E\big[
                            \big\langle
                                Dv^{n}(T-t,\tilde{Y}^{n}(t))
                                ,
                                P_{n}F(\tilde{Y}^{n}(t))
                                -
                                P_{n}F(Y_{m}^{n})
                            \big\rangle
                    \big]
              \diff{t}
              \\&~~~~~+
              \int_{t_{m}}^{t_{m+1}}
                  \E\big[
                            \big\langle
                                Dv^{n}(T-t,\tilde{Y}^{n}(t))
                                ,
                                \big(I - E_{n}(t-t_{m})\big)
                                P_{n}F(Y_{m}^{n})
                            \big\rangle
                    \big]
              \diff{t}
               \\&~~+
               \frac{1}{2}
               \int_{t_{m}}^{t_{m+1}}
                   \E\big[
                            \tr
                            \big\{
                                     D^{2}v^{n}(T-t,\tilde{Y}^{n}(t))
                                     \big(I - E_{n}(t-t_{m})\big)
                                     \big(P_{n}Q^{\frac{1}{2}}\big)
                                     \big(P_{n}Q^{\frac{1}{2}}\big)^{*}
                            \big\}
                     \big]
               \diff{t}
               \\&~~+
               \frac{1}{2}
               \int_{t_{m}}^{t_{m+1}}
                   \E\big[
                            \tr
                            \big\{
                                     D^{2}v^{n}(T-t,\tilde{Y}^{n}(t))
                                     E_{n}(t-t_{m})
                                     \big(P_{n}Q^{\frac{1}{2}}\big)
                                     \big((I-E_{n}(t-t_{m}))
                                     (P_{n}Q^{\frac{1}{2}})\big)^{*}
                            \big\}
                     \big]
               \diff{t}
     \\:=&\sum\limits_{m=0}^{M-1} J_{1}^{m}+J_{2}^{m}+J_{3}^{m}+J_{4}^{m}.
      \end{split}
      \end{equation}
      Below we will estimate these terms separately. For $J_{1}^{m}$, further decomposition leads to
      \begin{equation}
            \label{eq:J.one}
      \begin{split}
            |J_{1}^{m}|
      \leq&
            \Big|
            \int_{t_{m}}^{t_{m+1}}
                \E\big[
                          \big\langle
                          Dv^{n}(T-t,\tilde{Y}^{n}(t))
                          -
                          Dv^{n}(T-t,Y_{m}^{n})
                          ,
                          P_{n}F(\tilde{Y}^{n}(t))
                          -
                          P_{n}F(Y_{m}^{n})
                          \big\rangle
                  \big]
            \diff{t}
            \Big|
            \\&+
            \Big|
            \int_{t_{m}}^{t_{m+1}}
                \E\big[\big\langle
                          Dv^{n}(T-t,Y_{m}^{n})
                          ,
                           P_{n}F(\tilde{Y}^{n}(t))
                           -
                           P_{n}F(Y_{m}^{n})
                  \big\rangle\big]
            \diff{t}
            \Big|
      :=J_{11}^{m}+J_{12}^{m}.
      \end{split}
      \end{equation}
      Applying Taylor's formula in Banach space, \eqref{eq:kol.solu.SECOND} with
      $\gamma_{1}=0,\gamma_{2}=0$, \eqref{eq:F.derivative.one} and \eqref{eq:continuous.extension.approximation} to $J_{11}^{m}$, we get
      \begin{equation}
            \label{eq:J.one.one}
      \begin{split}
            J_{11}^{m}
      \leq&
            \int_{t_{m}}^{t_{m+1}}\int_{0}^{1}
                \E\big[
                          \big|\big\langle
                          D^{2}v^{n}(T-t,\chi(r))
                          (\tilde{Y}^{n}(t)-Y_{m}^{n})
                           ,
                          P_{n}F(\tilde{Y}^{n}(t))
                          -
                          P_{n}F(Y_{m}^{n})
                          \big\rangle\big|
                         \big]
             \diff{r}\diff{t}
      \\\leq&
            C
            \int_{t_{m}}^{t_{m+1}}
                (1+(T-t)^{-\eta})e^{-\tilde{c}(T-t)}
                \E\big[
                          \|\tilde{Y}^{n}(t)-Y_{m}^{n}\|
                          \|
                          P_{n}F(\tilde{Y}^{n}(t))
                          -
                          P_{n}F(Y_{m}^{n})
                          \|
                  \big]
            \diff{t}
      \\\leq&
            C
            \int_{t_{m}}^{t_{m+1}}
                (1+(T-t)^{-\eta})e^{-\tilde{c}(T-t)}
                \E\big[
                  \|\tilde{Y}^{n}(t)-Y_{m}^{n}\|^{2}
                  \big]
            \diff{t}
      \\\leq&
            C{\tau}^{\beta-\epsilon}
            \int_{t_{m}}^{t_{m+1}}
                (1+(T-t)^{-\eta})e^{-\tilde{c}(T-t)}
            \diff{t},
      \end{split}
      \end{equation}
      where $\chi(r):=Y_{m}^{n} + r(\tilde{Y}^{n}(t)-Y_{m}^{n}), \forall\,r \in [0,1]$.
      Using Taylor's formula in Banach space again further decomposes $J_{12}^{m}$ as follows
      \begin{equation}
            \label{eq:J.one.two}
      \begin{split}
            J_{12}^{m}
      \leq&
            \Big|
            \int_{t_{m}}^{t_{m+1}}
                \E\big[
                          \big\langle
                          Dv^{n}(T-t,Y_{m}^{n})
                          ,
                          P_{n}F'(Y_{m}^{n})
                          (\tilde{Y}^{n}(t)-Y_{m}^{n})
                          \big\rangle
                  \big]
            \diff{t}
            \Big|
      \\&+
            \Big|
            \int_{t_{m}}^{t_{m+1}}
                \E\big[
                          \big\langle
                          Dv^{n}(T-t,Y_{m}^{n})
                          ,
                          \int_{0}^{1}
                               P_{n}F''
                               (Y_{m}^{n}+r(\tilde{Y}^{n}(t)-Y_{m}^{n}))
                               \\&~~~
                               (
                                   \tilde{Y}^{n}(t)-Y_{m}^{n}
                                   ,
                                   \tilde{Y}^{n}(t)-Y_{m}^{n}
                               )
                               (1-r)
                          \diff{r}
                          \big\rangle
                  \big]
            \diff{t}
            \Big|
      :=J_{12}^{ma}+J_{12}^{mb}.
      \end{split}
      \end{equation}
      By the Cauchy--Schwarz inequality, \eqref{eq:kol.solu.FIRST}, \eqref{eq:F.derivative.one.delta.n} and \eqref{eq:F.derivative.one}, we can derive from \eqref{eq:tildeYntYmn} that
      \begin{equation*}\label{eq:J.one.two.a.step.one}
      \begin{split}
            J_{12}^{ma}
      \leq&
            \Big|
            \int_{t_{m}}^{t_{m+1}}
                \E\big[
                          \big\langle
                          Dv^{n}(T-t,Y_{m}^{n})
                          ,
                           P_{n}F'(Y_{m}^{n})
                           (E_{n}(t-t_{m})-I)Y_{m}^{n}
                          \big\rangle
                  \big]
            \diff{t}
            \Big|
      \\&+
            \Big|
            \int_{t_{m}}^{t_{m+1}}
                \E\big[
                          \big\langle
                          Dv^{n}(T-t,Y_{m}^{n})
                          ,
                              P_{n}F'(Y_{m}^{n})
                                  E_{n}(t-t_{m})
                                  P_{n}F(Y_{m}^{n})
                                  (t-t_{m})
                          \big\rangle
                  \big]
            \diff{t}
            \Big|
      \\\leq&
            C\int_{t_{m}}^{t_{m+1}}
                 \big(1+(T-t)^{-\frac{\delta}{2}}\big)
                 e^{-\tilde{c}(T-t)}
                 \E\big[
                   \|
                          (-A_{n})^{-\frac{\delta}{2}}
                          P_{n}F'(Y_{m}^{n})
                          (E_{n}(t-t_{m})-I)Y_{m}^{n}
                   \|
                   \big]
             \diff{t}
      \\&+
            C\tau
            \int_{t_{m}}^{t_{m+1}}
                e^{-\tilde{c}(T-t)}
                \E\big[
                  \|
                          P_{n}F'(Y_{m}^{n})
                              E_{n}(t-t_{m})
                              P_{n}F(Y_{m}^{n})
                  \|
                  \big]
            \diff{t}
      \\\leq&
            C\int_{t_{m}}^{t_{m+1}}
                 \big(1+(T-t)^{-\frac{\delta}{2}}\big)
                 e^{-\tilde{c}(T-t)}
                 \E\big[
                   (1+\|Y_{m}^{n}\|_{1})
                   \|(E_{n}(t-t_{m})-I)Y_{m}^{n}\|_{-1}
                   \big]
             \diff{t}
      \\&+
            C\tau
            \int_{t_{m}}^{t_{m+1}}
                e^{-\tilde{c}(T-t)}
                \E\big[
                  \|E_{n}(t-t_{m})P_{n}F(Y_{m}^{n})\|
                  \big]
            \diff{t}.
      \end{split}
      \end{equation*}
      Applying H\"older's inequality, elementary inequality, the stability of the
      semigroup $\{E_{n}(t)\}_{t \geq 0}$ and \eqref{eq:bound.momnet.two},
      we can deduce that
      \begin{equation}
            \label{eq:J.one.two.a.step.two}
      \begin{split}
            J_{12}^{ma}
      \leq&
            C\int_{t_{m}}^{t_{m+1}}
                 \big(1+(T-t)^{-\frac{\delta}{2}}\big)
                 e^{-\tilde{c}(T-t)}
                 \big(
                     1
                     +
                     \big(
                         \E\big[\|Y_{m}^{n}\|_{1}^{2}\big]
                     \big)^{\frac{1}{2}}
                 \big)
                 \big(
                 \E\big[
                   \|
                     (-A_{n})^{\frac{\beta-\epsilon}{2}}Y_{m}^{n}
                   \|^{2}
                   \big]
                 \big)^{\frac{1}{2}}
                 \\&
                 \|
                    (-A_{n})^{-\frac{1+\beta-\epsilon}{2}}
                    (E_{n}(t-t_{m})-I)
                 \|_{\mathcal{L}(H_{n})}
             \diff{t}
            +
            C\tau
            \int_{t_{m}}^{t_{m+1}}
                e^{-\tilde{c}(T-t)}
            \diff{t}
      \\\leq&
            C{\tau}^{\beta-\epsilon}
             \int_{t_{m}}^{t_{m+1}}
                 \big(1+(T-t)^{-\frac{\delta}{2}}\big)
                 e^{-\tilde{c}(T-t)}
             \diff{t}
            +
            C\tau
            \int_{t_{m}}^{t_{m+1}}
                e^{-\tilde{c}(T-t)}
            \diff{t}
      \\\leq&
            C{\tau}^{\beta-\epsilon}
             \int_{t_{m}}^{t_{m+1}}
                 \big(1+(T-t)^{-\frac{\delta}{2}}\big)
                 e^{-\tilde{c}(T-t)}
             \diff{t},
      \end{split}
      \end{equation}
      where \eqref{eq:semigroup.smooth.two}, Lemmas \ref{lem:bound.momnet} and \ref{lem:bound.momnet.1/2} were employed in the second step.
      Thanks to \eqref{eq:kol.solu.FIRST},
      \eqref{eq:F.derivative.second.n}
      and \eqref{eq:continuous.extension.approximation}, one gets
      \begin{equation}\label{eq:J.one.two.b}
      \begin{split}
            J_{12}^{mb}
      \leq&
            C
            \int_{t_{m}}^{t_{m+1}}
            \int_{0}^{1}
                \E\big[
                          \|
                          (-A_{n})^{-\eta}
                          P_{n}F''
                          (Y_{m}^{n}+r(\tilde{Y}^{n}(t)-Y_{m}^{n}))
                          \\&~~~
                          (
                                \tilde{Y}^{n}(t)-Y_{m}^{n}
                                ,
                                \tilde{Y}^{n}(t)-Y_{m}^{n}
                          )
                          \|
                  \big]
                          \big(1+(T-t)^{-\eta}\big)
                          e^{-\tilde{c}(T-t)}(1-r)
            \diff{r}
            \diff{t}
      \\\leq&
            C
            \int_{t_{m}}^{t_{m+1}}
                \E\big[
                  \|\tilde{Y}^{n}(t)-Y_{m}^{n}\|^{2}
                  \big]
                 \big(1+(T-t)^{-\eta}\big)
                 e^{-\tilde{c}(T-t)}
            \diff{t}
      \\\leq&
            C{\tau}^{\beta-\epsilon}
            \int_{t_{m}}^{t_{m+1}}
                \big(1+(T-t)^{-\eta}\big)
                e^{-\tilde{c}(T-t)}
            \diff{t}.
      \end{split}
      \end{equation}
      Putting \eqref{eq:J.one.two.a.step.two}--\eqref{eq:J.one.two.b} into \eqref{eq:J.one.two} implies
      \begin{equation*}
            \label{eq:J.one.two.outcome}
            J_{12}^{m}
      \leq
            C{\tau}^{\beta-\epsilon}
            \int_{t_{m}}^{t_{m+1}}
                \big(1+(T-t)^{-\frac{\delta}{2}}+(T-t)^{-\eta}\big)
                e^{-\tilde{c}(T-t)}
            \diff{t},
      \end{equation*}
      which together with \eqref{eq:J.one}--\eqref{eq:J.one.one}
      leads to
      \begin{equation}
            \label{eq:J.one.outcome}
            |J_{1}^{m}|
      \leq
            C{\tau}^{\beta-\epsilon}
            \int_{t_{m}}^{t_{m+1}}
                \big(1+(T-t)^{-\frac{\delta}{2}}+(T-t)^{-\eta}\big)
                e^{-\tilde{c}(T-t)}
            \diff{t}.
      \end{equation}
      As to $J_{2}^{m}$, with the help of \eqref{eq:kol.solu.FIRST},
      \eqref{eq:bound.momnet.two} and \eqref{eq:semigroup.smooth.two},
      we can conclude that
      \begin{equation}
            \label{eq:J.two.outcome}
      \begin{split}
            |J_{2}^{m}|
      \leq&
            \int_{t_{m}}^{t_{m+1}}
                \E\big[
                       \|
                          (-A_{n})^{1-\epsilon}
                          Dv^{n}(T-t,\tilde{Y}^{n}(t))
                       \|
                        \\&~~~~~~~~~\cdot
                       \|
                          (-A_{n})^{-(1-\epsilon)}
                          (I - E_{n}(t-t_{m}))
                       \|_{\mathcal{L}_{H_{n}}}
                          \cdot
                       \|P_{n}F(Y_{m}^{n})\|
                       \big]
            \diff{t}
      \\\leq&
            C{\tau}^{1-\epsilon}
            \int_{t_{m}}^{t_{m+1}}
                \big(1+(T-t)^{-(1-\epsilon)}\big)e^{-\tilde{c}(T-t)}
            \diff{t}.
      \end{split}
      \end{equation}
      Concerning $J_{3}^{m}$, we employ \eqref{eq:.operator.Gamma.one}, \eqref{eq:.operator.Gamma.three} and the self-adjointness of $A_{n}$ to obtain
      \begin{equation*}
      \begin{split}
            |J_{3}^{m}|
      =&
            \Big|
            \frac{1}{2}
            \int_{t_{m}}^{t_{m+1}}
                \E\big[
                         \tr
                         \big\{
                                 (-A_{n})^{\frac{1-\beta}{2}}
                                 D^{2}v^{n}(T-t,\tilde{Y}^{n}(t))
                                 \big(I - E_{n}(t-t_{m})\big)
                                 \\&\cdot
                                 \big(P_{n}Q^{\frac{1}{2}}\big)
                                 \big((-A_{n})^{\frac{\beta-1}{2}}
                                        P_{n}Q^{\frac{1}{2}}\big)^{*}
                         \big\}
                  \big]
            \diff{t}
            \Big|
      \\\leq&
            \frac{1}{2}
            \int_{t_{m}}^{t_{m+1}}
                \E\big[
                         \big\|
                                 (-A_{n})^{\frac{1-\beta}{2}}
                                 D^{2}v^{n}(T-t,\tilde{Y}^{n}(t))
                                 (-A_{n})^{\frac{1+\beta}{2}-\epsilon}
                         \big\|_{\mathcal{L}(H_{n})}
      \\&\cdot
                         \big\|
                                 (-A_{n})^{-\frac{1+\beta}{2}+\epsilon}
                                 \big(I - E_{n}(t-t_{m})\big)
                                 \big(P_{n}Q^{\frac{1}{2}}\big)
                                 \big((-A_{n})^{\frac{\beta-1}{2}}
                                       P_{n}Q^{\frac{1}{2}}\big)^{*}
                         \big\|_{\mathcal{L}_{1}(H_{n})}
                  \big]
            \diff{t}
      \\\leq&
            \frac{1}{2}
            \int_{t_{m}}^{t_{m+1}}
                \E\big[
                  \big\|
                      (-A_{n})^{\frac{1-\beta}{2}}
                      D^{2}v^{n}(T-t,\tilde{Y}^{n}(t))
                      (-A_{n})^{\frac{1+\beta}{2}-\epsilon}
                  \big\|_{\mathcal{L}(H_{n})}
      \\&~~~~~~~~~~~~~\cdot
            \|
                (-A_{n})^{-(\beta-\epsilon)}
                \big(I - E_{n}(t-t_{m})\big)
            \|_{\mathcal{L}(H_{n})}
      \\&~~~~~~~~~~~~~\cdot
            \big\|
            \big((-A_{n})^{\frac{\beta-1}{2}}
                   P_{n}Q^{\frac{1}{2}}\big)
            \big((-A_{n})^{\frac{\beta-1}{2}}
                   P_{n}Q^{\frac{1}{2}}\big)^{*}
            \big\|_{\mathcal{L}_{1}(H_{n})}
            \big]
            \diff{t}.
      \end{split}
      \end{equation*}
      By \eqref{eq:.operator.Gamma.two}, \eqref{eq:kol.solu.SECOND},
      \eqref{eq:semigroup.smooth.two} and \eqref{eq:A.and.Q.n}, it follows that
      \begin{equation}\label{eq:J.three.outcome}
      \begin{split}
            |J_{3}^{m}|
      \leq&
            \frac{1}{2}
            \int_{t_{m}}^{t_{m+1}}
                \E\big[
                  \big\|
                      (-A_{n})^{\frac{1-\beta}{2}}
                      D^{2}v^{n}(T-t,\tilde{Y}^{n}(t))
                      (-A_{n})^{\frac{1+\beta}{2}-\epsilon}
                  \big\|_{\mathcal{L}(H_{n})}
      \\&\cdot
            \big\|
                (-A_{n})^{-(\beta-\epsilon)}
                \big(I - E_{n}(t-t_{m})\big)
            \big\|_{\mathcal{L}(H_{n})}
            \big\|
                (-A_{n})^{\frac{\beta-1}{2}}
                  P_{n}Q^{\frac{1}{2}}
            \big\|_{\mathcal{L}_{2}(H,H_{n})}^{2}
            \big]
            \diff{t}.
%
      \\\leq&
            C{\tau}^{\beta-\epsilon}
            \int_{t_{m}}^{t_{m+1}}
                \big(1+(T-t)^{-\eta}+(T-t)^{-(1-\epsilon)}\big)
                e^{-\tilde{c}(T-t)}
            \diff{t}.
      \end{split}
      \end{equation}
      With regard to $J_{4}^{m}$, similarly to $J_{3}^{m}$, we can get
      \begin{equation}
            \label{eq:J.four.outcome}
      \begin{split}
            |J_{4}^{m}|
      \leq&
            C{\tau}^{\beta-\epsilon}
            \int_{t_{m}}^{t_{m+1}}
                \big(1+(T-t)^{-\eta}+(T-t)^{-(1-\epsilon)}\big)
                e^{-\tilde{c}(T-t)}
            \diff{t}.
      \end{split}
      \end{equation}
      Inserting \eqref{eq:J.one.outcome}, \eqref{eq:J.two.outcome},
      \eqref{eq:J.three.outcome} and \eqref{eq:J.four.outcome}
      into \eqref{eq:telescoping} and using \eqref{eq:Gamma}
      yields the required conclusion.
\end{proof}

\subsection{Error of invariant measures for the space-time full discretization}
            \label{sec:order.full}

\begin{thm}\label{th:order.measure.temporal}
       Suppose that Assumptions \ref{ass:main.ass} and \ref{ass:main.ass2} hold. Let $\tau < \tau_0 \leq \frac{\lambda_1-L_F}{4L^2}$ and let $\nu^{n}$ and $\nu_{\tau}^{n}$ be the corresponding unique invariant measure of $\{ X^{n}(t) \}_{t \geq 0}$ and $\{Y_{m}^{n}\}_{m \in \N}$, respectively. Then for any 
       $\Phi \in C_{b}^{2}(H,\R)$ there exists $C>0$ independent of $n, \tau$ such that
      \begin{equation}\label{eq:order.measure.temporal}
            \Big|
            \int_{H_{n}} \Phi(y) \,\nu^{n}(\dif{y})
            -
            \int_{H_{n}} \Phi(y) \,\nu_{\tau}^{n}(\dif{y})
            \Big|
            \leq
            C{\tau}^{\beta-\epsilon}.
      \end{equation}
\end{thm}

\begin{proof}
      Theorem \ref{th:ergodic.three} and the definition of ergodicity imply
      \begin{equation}\label{eq:ergodic.three}
            \lim\limits_{M\to\infty}
            \frac{1}{M}
            \sum\limits_{m=0}^{M-1}
            \E\big[\Phi(Y_{m}^{n})\big]
            =
            \int_{H_{n}} \Phi(y) \,\nu_{\tau}^{n}(\dif{y}),
            \quad
            \forall\, \Phi \in C_{b}^{2}(H,\R),
      \end{equation}
      which in combination with \eqref{eq:ergodic.two} and \eqref{eq:temporal.order.I} results in
      \begin{equation*}
      \begin{split}
            \Big|
            \int_{H_{n}} \Phi(y)& \,\nu^{n}(\dif{y})
            -
            \int_{H_{n}} \Phi(y) \,\nu_{\tau}^{n}(\dif{y})
            \Big|
      \leq
            \lim\limits_{M\to\infty}\frac{1}{M\tau}
            \sum\limits_{m=0}^{M-1}
            \int_{t_{m}}^{t_{m+1}}
            \big|
                    \E[\Phi(X^{n}(t))]
                    -
                    \E[\Phi(Y_{m}^{n})]
            \big|
            \diff{t}
      \\\leq&
            \lim\limits_{M\to\infty}\frac{1}{M\tau}
            \sum\limits_{m=0}^{M-1}
            \int_{t_{m}}^{t_{m+1}}
            \big|
                    \E[\Phi(X^{n}(t))]
                    -
                    \E[\Phi(X^{n}(t_{m}))]
            \big|
            \diff{t}
      +
            C{\tau}^{\beta-\epsilon}
      :=
           K_{1} + C{\tau}^{\beta-\epsilon}.
      \end{split}
      \end{equation*}
      Now it remains to treat $K_{1}$.
      Using \eqref{eq:kol.solu}--\eqref{eq:kol}, we can show that for any $t \in [t_{m},t_{m+1}]$,
      \begin{equation*}
      \begin{split}
            &
            \E\big[\Phi(X^{n}(t))\big]
            -
            \E\big[\Phi(X^{n}(t_{m}))\big]
      =
            v^{n}(t,X_{0}^{n})-v^{n}(t_{m},X_{0}^{n})
      =
            \int_{t_{m}}^{t}
                \frac{\partial{v^{n}(s,X_{0}^{n})}}{\partial{s}}
            \diff{s}
      \\=&
            \int_{t_{m}}^{t}
                \big\langle
                Dv^{n}(s,X_{0}^{n})
                ,
                A_{n}X_{0}^{n}+P_{n}F(X_{0}^{n})
                \big\rangle
                +
                \frac{1}{2}
                \tr\big\{
                   D^{2}v^{n}(s,X_{0}^{n})
                   (P_{n}Q^{\frac{1}{2}})
                   (P_{n}Q^{\frac{1}{2}})^{*}
                   \big\}
            \diff{s}
      \\=&
            -\int_{t_{m}}^{t}
                 \big\langle
                    (-A_{n})^{1-\frac{\beta}{2}}Dv^{n}(s,X_{0}^{n})
                    ,
                    (-A_{n})^{\frac{\beta}{2}}X_{0}^{n}
                 \big\rangle
             \diff{s}
            +
            \int_{t_{m}}^{t}
                \big\langle
                    Dv^{n}(s,X_{0}^{n}) , P_{n}F(X_{0}^{n})
                \big\rangle
            \diff{s}
            \\&+
            \frac{1}{2}
            \int_{t_{m}}^{t}
                \tr\big\{
                   (-A_{n})^{1-\beta}
                   D^{2}v^{n}(s,X_{0}^{n})
                   \big(
                       (-A_{n})^{\frac{\beta-1}{2}}P_{n}Q^{\frac{1}{2}}
                   \big)
                   \big(
                   (-A_{n})^{\frac{\beta-1}{2}}P_{n}Q^{\frac{1}{2}}
                   \big)^{*}
                   \big\}
            \diff{s}.
      \end{split}
      \end{equation*}
      By \eqref{eq:kol.solu.FIRST}--\eqref{eq:kol.solu.SECOND},
      \eqref{eq:F.derivative.one}, \eqref{eq:A.and.Q.n} and $X_{0} \in \dot{H}^{\beta}$,
      we have
      \begin{equation*}
      \begin{split}
            &
            \big|
            \E[\Phi(X^{n}(t))]
            -
            \E[\Phi(X^{n}(t_{m}))]
            \big|
      \\\leq&
            \int_{t_{m}}^{t}
                \|(-A_{n})^{1-\frac{\beta}{2}}Dv^{n}(s,X_{0}^{n})\|
                \|(-A_{n})^{\frac{\beta}{2}}X_{0}^{n}\|
            \diff{s}
            +
            \int_{t_{m}}^{t}
                \| Dv^{n}(s,X_{0}^{n}) \| \| P_{n}F(X_{0}^{n})\|
            \diff{s}
            \\&+
            \frac{1}{2}
            \int_{t_{m}}^{t}
                \|(-A_{n})^{1-\beta}D^{2}v^{n}(s,X_{0}^{n})\|_{\mathcal{L}(H_{n})}
                \|
                 (-A_{n})^{\frac{\beta-1}{2}}P_{n}Q^{\frac{1}{2}}
                \|_{\mathcal{L}_{2}(H,H_{n})}^{2}
            \diff{s}
      \\\leq&
            C\|(-A_{n})^{\frac{\beta}{2}}X_{0}^{n}\|
            \int_{t_{m}}^{t}
                \big( 1+s^{\frac{\beta}{2}-1} \big)e^{-\tilde{c}s}
            \diff{s}
            +
            CL(1+\|X_{0}^{n}\|) \int_{t_{m}}^{t}   e^{-\tilde{c}s}  \diff{s}
            \\&+
            C\big\|
                 (-A_{n})^{\frac{\beta-1}{2}}P_{n}Q^{\frac{1}{2}}
             \big\|_{\mathcal{L}_{2}(H,H_{n})}^{2}
             \int_{t_{m}}^{t}
                 \big(1+s^{-\eta}+s^{\beta-1}\big) e^{-\tilde{c}s}
             \diff{s}
      \\\leq&
            C\int_{t_{m}}^{t_{m+1}}
             \big(1+s^{-\eta}+s^{\frac{\beta}{2}-1}+s^{\beta-1}\big) e^{-\tilde{c}s}
             \diff{s}.
      \end{split}
      \end{equation*}
      With this and \eqref{eq:Gamma}, we can easily get $K_{1}=0$ and hence complete the proof.
\end{proof}

\begin{rek}\label{re:II}
  Bearing Remark \ref{re:I} in mind and specializing Theorem \ref{th:order.measure.temporal} to the space-time white noise case with $\beta<\frac{1}{2}$ yields that the convergence order between $\nu^n$ and $\nu_{\tau}^n$ is $\frac{1}{2}-\epsilon$ for arbitrarily small $\epsilon > 0$, which coincides with that in \cite{brehier2014approximation} for the linear implicit Euler scheme. Further applying this theorem to the trace class noise case with $\beta=1$ gives an order $1-\epsilon$ with arbitrarily small $\epsilon > 0$ for the convergence rate between $\nu^n$ and $\nu_{\tau}^n$ in space dimension $d = 1$.
\end{rek}

As a direct consequence of Theorems \ref{th:order.measure.spatial} and \ref{th:order.measure.temporal}, we have
\begin{coy}\label{Corollary:full-discrete}
Suppose that Assumptions \ref{ass:main.ass} and \ref{ass:main.ass2} hold. Let $\tau < \tau_0 \leq \frac{\lambda_1-L_F}{4L^2}$
and let $\nu$ and $\nu_{\tau}^{n}$ be the corresponding unique invariant measure of $\{X(t)\}_{t \geq 0}$ and
$\{Y_{m}^{n}\}_{m \in \N}$, respectively. Then for any $\Phi \in C_{b}^{2}(H,\R)$
there exists $C>0$ independent of $n, \tau $ such that
\begin{equation}\label{eq:order.measure.space-time}
            \Big|
            \int_{H} \Phi(y) \,\nu(\dif{y})
            -
            \int_{H_{n}} \Phi(y) \,\nu_{\tau}^{n}(\dif{y})
            \Big|
            \leq
            C
            (
            \lambda_{n}^{-\beta+\epsilon}
            +
            {\tau}^{\beta-\epsilon}
            ).
      \end{equation}
\end{coy}

\section{Numerical experiments}
\label{sect:numer-exp}

In this section, some numerical experiments are performed to illustrate the previous findings.
We consider an example from \cite[Example~3.2]{wang2016weak} as follows
      \begin{equation}
            \label{eq:example}
      \left\{
      \begin{array}{ll}
            \hspace{-0.5em}
            \frac{\partial{u}}{\partial{t}}
            =
            \frac{\partial^{2}{u}}{\partial{x^{2}}}
            +
            1 + u + \sin(u)
            +
            {\dot{W}^{Q}},
            &
            t > 0,~x \in (0,1),
            \\
            \hspace{-0.5em}
            u(0,x) = \sqrt{2}\sin(\pi{x}),
            &
            x \in (0,1),
            \\
            \hspace{-0.5em}
            u(t,0) = u(t,1) =0,
            &
            t > 0.
      \end{array}
      \right.
      \end{equation}
In order to fulfill \eqref{eq:A.and.Q} and \eqref{eq:W.Q.t}, we take $ q_{i} = 1, i \in \N$, $\beta < \frac{1}{2}$ for the space-time white noise case ($Q = I$) and $q_{i} = i^{-1.005}, i \in \N$, $\beta = 1$ for the trace class noise case ($\tr(Q) < \infty$).
Then one can easily show that all conditions in Assumptions \ref{ass:main.ass} and \ref{ass:main.ass2} are satisfied in this setting.
We also remark that all the expectations are approximated by computing averages over 100 samples and the exact solutions to \eqref{eq:example} are identified with the numerical ones using a large $n = 2^{10}$ as reference for the spatial test and a small $\tau = 2^{-15}$ as reference for the temporal test.



By ergodicity, we know that the temporal averages $\frac{1}{M+1}\sum_{m=0}^{M}\E[\Phi(Y_{m}^{n})]$ should be a constant for all initial values in the whole space and may vary for different test functions $\Phi \in C_{b}^{2}(H,\R)$. These facts are numerically verified by Table \ref{averageinitial} with three different initial values $u_0^1,u_0^2,u_0^3$ and Table \ref{averagetest} with three different test functions $\Phi_1,\Phi_2,\Phi_3$. Additionally, both the spatial and temporal weak errors listed in Table \ref{tab:weakerrors}  show that
these errors are independent of time $T$.

Next we test the weak convergence orders with $u_{0}(x) = \sqrt{2}\sin(\pi{x}), x \in (0,1)$ being the initial value.
To this end, we take $\tau = 2^{-20}, n= 2^{-i},i=1,2,\ldots,7$ for the spatial test and $n= 100, \tau = 2^{-j}, j= 5,6,\ldots,12$ for the temporal test.
We mention that we choose $\Phi(y) = \exp(-|y|^{2}), y \in \R^n$ to be the test function and set the final time $T = 20$, which is large enough to ensure that the equilibrium is reached based on Tables \ref{averageinitial} and  \ref{averagetest}.
From Figure \ref{fig:error.order}, one can observe that, the slopes of the error lines and the reference lines match well, indicating that the convergence order is $1-\epsilon$ in space and $\frac{1}{2}-\epsilon$ in time for the space-time white noise case and $2-\epsilon$ in space and $1-\epsilon$ in time for the trace class noise case with arbitrarily small $\epsilon>0$.

Finally we fix $n = 100$ and also compare weak errors of the exponential Euler scheme with those of the existing linear implicit Euler scheme in\cite{brehier2014approximation,brehier2017approximation}. From Table \ref{tab:temporalweakerrors}, we can see that the exponential Euler scheme is always considerably more accurate than the linear implicit Euler scheme.
%
%
\newcommand{\bb}[1]{\raisebox{-3.5ex}[0pt][0pt]{\shortstack{#1}}}
\begin{center}
\begin{table}
\centering
\begin{threeparttable}
  \caption{The temporal averages
  for different initial values}\label{averageinitial}
\begin{tabular}{p{0.5cm}<{\centering}
  |  p{1.5cm}<{\centering}  p{1.5cm}<{\centering}  p{1.5cm}<{\centering}
  |  p{1.5cm}<{\centering}  p{1.5cm}<{\centering}  p{1.5cm}<{\centering} }\hline
  \multicolumn{7}{c}
  {\parbox[c][9mm]{0pt}{}
  $
  \frac{1}{M+1}\sum_{m=0}^{M}\E[\Phi(Y_{m}^{n})],
   n = 100, M = \tfrac{T}{\tau}, \tau = 2^{-6},
   \Phi(y) = e^{-|y|^{2}}, y \in \R^n
  $
  }\\\hline
  \raisebox{-4ex}[0pt][0pt]{\shortstack{$T$}}
  &
  \multicolumn{6}{c}
  {  \parbox[c][6mm]{0pt}{}
     $u_{0}^{1}(x) = 0$,
     $u_{0}^{2}(x) = \sqrt{2}\sin(\pi{x})$,
     $u_{0}^{3}(x) = \sum_{i=1}^{\infty}\sin(i\pi{x})/i$
  }
  \\\cline{2-7}
  & \multicolumn{3}{c|}{\parbox[c][6mm]{0pt}{} $Q = I$}
  & \multicolumn{3}{c}{\parbox[c][6mm]{0pt}{} $\tr(Q) < \infty$}
  \\\cline{2-7}
  & \multicolumn{1}{p{1.5cm}<{\centering}}{$u_{0}^{1}$}
  & \multicolumn{1}{p{1.5cm}<{\centering}}{$u_{0}^{2}$}
  & \multicolumn{1}{p{1.5cm}<{\centering}|}{$u_{0}^{3}$}
  & \multicolumn{1}{p{1.5cm}<{\centering}}{$u_{0}^{1}$}
  & \multicolumn{1}{p{1.5cm}<{\centering}}{$u_{0}^{2}$}
  & \multicolumn{1}{p{1.5cm}<{\centering}}{$u_{0}^{3}$}
  \\\hline
  $10$  &0.93451 &0.93095 &0.93187 &0.93828 &0.93471 &0.93563 \\
  $20$  &0.93553 &0.93375 &0.93421 &0.93932 &0.93753 &0.93799 \\
  $50$  &0.93495 &0.93424 &0.93442 &0.93875 &0.93803 &0.93821 \\
  $100$ &0.93506 &0.93471 &0.93480 &0.93885 &0.93850 &0.93859 \\
  $200$ &0.93482 &0.93465 &0.93469 &0.93862 &0.93844 &0.93848 \\
  $500$ &0.93523 &0.93516 &0.93518 &0.93902 &0.93895 &0.93897 \\
  \hline
\end{tabular}
\end{threeparttable}
\end{table}
\end{center}
\vspace{-3em}
\begin{center}
\begin{table}
\centering
\begin{threeparttable}
  \caption{The temporal averages
  for different test functions}\label{averagetest}
\begin{tabular}{p{0.5cm}<{\centering}
  |  p{1.5cm}<{\centering}  p{1.5cm}<{\centering}  p{1.5cm}<{\centering}
  |  p{1.5cm}<{\centering}  p{1.5cm}<{\centering}  p{1.5cm}<{\centering} }\hline
  \multicolumn{7}{c}
  {\parbox[c][9mm]{0pt}{}
  $
  \frac{1}{M+1}\sum_{m=0}^{M}\E[\Phi(Y_{m}^{n})],
   n = 100, M = \tfrac{T}{\tau}, \tau = 2^{-6},
   u_{0}(x) = \sqrt{2}\sin(\pi{x})
  $
  }\\\hline
  \raisebox{-4ex}[0pt][0pt]{\shortstack{$T$}}
  &
  \multicolumn{6}{c}
  {  \parbox[c][6mm]{0pt}{}
     $\Phi_{1}(y) = e^{-|y|^{2}}$,
     $\Phi_{2}(y) = \sin(|y|)$,
     $\Phi_{3}(y) = \cos(|y|)$, $y \in \R^n$
  }
  \\\cline{2-7}
  & \multicolumn{3}{c|}{\parbox[c][6mm]{0pt}{} $Q = I$}
  & \multicolumn{3}{c}{\parbox[c][6mm]{0pt}{} $\tr(Q) < \infty$}
  \\\cline{2-7}
  & \multicolumn{1}{p{1.5cm}<{\centering}}{$\Phi_{1}$}
  & \multicolumn{1}{p{1.5cm}<{\centering}}{$\Phi_{2}$}
  & \multicolumn{1}{p{1.5cm}<{\centering}|}{$\Phi_{3}$}
  & \multicolumn{1}{p{1.5cm}<{\centering}}{$\Phi_{1}$}
  & \multicolumn{1}{p{1.5cm}<{\centering}}{$\Phi_{2}$}
  & \multicolumn{1}{p{1.5cm}<{\centering}}{$\Phi_{3}$}
  \\\hline
  $10$  &0.93059 &0.22854 &0.96248 &0.93471 &0.21735 &0.96448 \\
  $20$  &0.93375 &0.22412 &0.96426 &0.93753 &0.21281 &0.96627 \\
  $50$  &0.93424 &0.22338 &0.96463 &0.93803 &0.21207 &0.96664 \\
  $100$ &0.93471 &0.22260 &0.96494 &0.93850 &0.21128 &0.96695 \\
  $200$ &0.93465 &0.22269 &0.96492 &0.93844 &0.21138 &0.96693 \\
  $500$ &0.93516 &0.22176 &0.96521 &0.93859 &0.21043 &0.96723 \\
  \hline
\end{tabular}
\end{threeparttable}
\end{table}
\end{center}
\vspace{-4em}
\begin{center}
\begin{table}
\centering
\begin{threeparttable}
  \caption{The spatial weak errors and the temporal weak errors}\label{tab:weakerrors}
\begin{tabular}{p{0.5cm}<{\centering}
  |  p{2.75cm}<{\centering}  p{2.75cm}<{\centering}
  |  p{2.75cm}<{\centering}  p{2.75cm}<{\centering}  }\hline
  \multicolumn{5}{c}
  {\parbox[c][9mm]{0pt}{} 
  $
   T = M\tau, u_{0}(x) = \sqrt{2}\sin(\pi{x}),
   \Phi(y) = e^{-|y|^{2}}, y \in \R^n
  $
  }\\\hline
  \raisebox{-4ex}[0pt][0pt]{\shortstack{$T$}}
  & \multicolumn{2}{c|}{\parbox[c][6mm]{0pt}{}
                       $\E[\Phi(X(T))] - \E[\Phi(X^{n}(T))]$
                       }
  & \multicolumn{2}{c}{\parbox[c][6mm]{0pt}{}
                      $\E[\Phi(X^{n}(T))] - \E[\Phi(Y_{M}^{n})]$
                      }
  \\\cline{2-5}
  & \multicolumn{2}{c|}{\parbox[c][6mm]{0pt}{}
                       $n = 50, n_{\text{ref}} = 100, \tau = 2^{-5}$
                       }
  & \multicolumn{2}{c}{\parbox[c][6mm]{0pt}{}
                      $n= 100, \tau = 2^{-5}, \tau_{\text{ref}} = 2^{-8}$
                      }
  \\\cline{2-5}
  & \multicolumn{1}{p{2.75cm}<{\centering}}{$Q = I$}
  & \multicolumn{1}{p{2.75cm}<{\centering}|}{$\tr(Q) < \infty$}
  & \multicolumn{1}{p{2.75cm}<{\centering}}{$Q = I$}
  & \multicolumn{1}{p{2.75cm}<{\centering}}{$\tr(Q) < \infty$}
  \\\hline
  $10$  &0.0000043250 &0.0000043352 &0.0323918796 &0.0236526091 \\
  $20$  &0.0000032528 &0.0000032601 &0.0325885619 &0.0249577624 \\
  $50$  &0.0000030435 &0.0000030462 &0.0289957068 &0.0218740569 \\
  $100$ &0.0000025270 &0.0000025277 &0.0349042067 &0.0276291459 \\
  $200$ &0.0000035050 &0.0000035119 &0.0297790439 &0.0226310914 \\
  $500$ &0.0000029597 &0.0000029627 &0.0347515849 &0.0261828866 \\
  \hline
\end{tabular}
\end{threeparttable}
\end{table}
\end{center}
\begin{figure}[!htbp]
\begin{center}
      \subfigure[spatial weak convergence orders]{\includegraphics[width=0.45\textwidth]{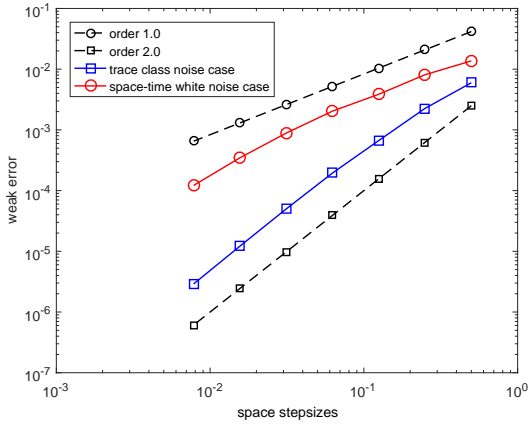}}
      \subfigure[temporal weak convergence orders]{\includegraphics[width=0.45\textwidth]{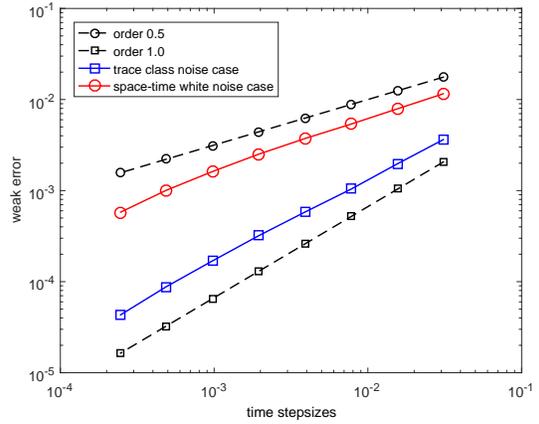}}
      \centering\caption{The weak convergence orders for trace class noise case and space-time white noise case}
      \label{fig:error.order}
\end{center}
\end{figure}
\begin{center}
\begin{table}
\centering
\begin{threeparttable}
  \caption{The temporal weak errors for exponential Euler (EE) scheme  and linear implicit Euler (LIE) scheme  with $n=100$, $\tau_{\text{ref}} = 2^{-15}$ and $T=20$}
\label{tab:temporalweakerrors}
\begin{tabular}{p{0.75cm}<{\centering}
  |  p{2.75cm}<{\centering}  p{2.75cm}<{\centering}
  |  p{2.75cm}<{\centering}  p{2.75cm}<{\centering}  }\hline
  \multicolumn{5}{c}
  {\parbox[c][9mm]{0pt}{} 
  $\E[\Phi(X^{n}(T))] - \E[\Phi(Y_{M}^{n})],
   M = T/\tau, u_{0}(x) = \sqrt{2}\sin(\pi{x}),
   \Phi(y) = e^{-|y|^{2}}, y \in \R^n
  $
  }\\\hline
  \raisebox{-1.5ex}[0pt][0pt]{\shortstack{$\tau$}}
  & \multicolumn{2}{c|}{\parbox[c][6mm]{0pt}{}
                       $Q = I$
                       }
  & \multicolumn{2}{c}{\parbox[c][6mm]{0pt}{}
                      $\tr(Q) < \infty$
                      }
  \\\cline{2-5}
  & \multicolumn{1}{p{2.75cm}<{\centering}}{EE}
  & \multicolumn{1}{p{2.75cm}<{\centering}|}{LIE}
  & \multicolumn{1}{p{2.75cm}<{\centering}}{EE}
  & \multicolumn{1}{p{2.75cm}<{\centering}}{LIE}
  \\\hline
  $2^{-5}$  &0.0458270571 &0.0723812355 &0.0270082268 &0.0559207587 \\
  $2^{-6}$  &0.0330721299 &0.0623557903 &0.0167672984 &0.0466522564 \\
  $2^{-7}$  &0.0227228795 &0.0482236813 &0.0094952116 &0.0334350770 \\
  $2^{-8}$  &0.0157090897 &0.0364990864 &0.0054211588 &0.0230548773 \\
  $2^{-9}$  &0.0106864989 &0.0260774425 &0.0030290821 &0.0145068919 \\
  $2^{-10}$ &0.0069976095 &0.0180278780 &0.0016330434 &0.0088346321 \\
  \hline
\end{tabular}
\end{threeparttable}
\end{table}
\end{center}

\noindent\textbf{Acknowledgments.} We are grateful to the two anonymous referees whose insightful
comments and valuable suggestions are crucial to the improvements of the manuscript.

\bibliographystyle{abbrv}

\end{document}